\documentclass[12pt,reqno]{amsart}

\usepackage{amsthm}

\newtheorem{lemma}{Lemma}
\newtheorem{theorem}{Theorem}
\newtheorem{proposition}{Proposition}
\newtheorem{corollary}{Corollary}
\newtheorem{conjecture}[equation]{Conjecture}
\newtheorem{remark}[equation]{Remark}

\newtheorem{question}[equation]{Question}
\newtheorem*{theorem*}{Theorem}
\newtheorem*{theoremA}{Theorem A}
\newtheorem*{theoremB}{Theorem B}
\newtheorem*{theoremC}{Theorem C}
\newtheorem*{conjecture2}{Conjecture 2}

\DeclareMathOperator{\Coeff}{Coeff}

\DeclareMathOperator{\kerr}{Ker}
\DeclareMathOperator{\imm}{Im}

\title{The Alon-Jaeger-Tarsi conjecture via group ring identities}

\begin{document}

\author{J\'anos Nagy}
\email{janomo4@gmail.com}
\address{Alfréd R\'enyi Institute of Mathematics,\newline \hspace*{4mm}
MTA-BME Lend\"ulet Arithmetic Combinatorics Research Group, E\"otv\"os Lor\'and Research Network}

\author{P\'eter P\'al Pach}
\thanks{The research supported by the Lend\"ulet program of the Hungarian Academy
  of Sciences (MTA). PPP was supported by the National Research, Development and Innovation
  Office NKFIH (Grant Nr. K124171, K129335 and BME NC TKP2020).
 }
\email{ppp@cs.bme.hu}
\address{MTA-BME Lend\"ulet Arithmetic Combinatorics Research Group,
  E\"otv\"os Lor\'and Research Network,
  Department of Computer Science and Information Theory, Budapest
  University of Technology and Economics, 1117 Budapest, Magyar tud\'osok
  k\"or\'utja 2, Hungary.}



\begin{abstract}
In this paper we resolve the Alon-Jaeger-Tarsi conjecture for sufficiently large primes. Namely, we show that for any finite field $\mathbb{F}$ of size $61<|\mathbb F|\ne 79$ and any nonsingular  matrix $M$ over $\mathbb{F}$ there exists a vector $x$ such that neither $x$ nor $Ax$ has a 0 component.

%In fact we prove a stronger statement. [...]
\end{abstract}

\maketitle

\linespread{1.2}

%\date{}

\section{Introduction}

%%%%% elejére majd még valami általános szöveg, polinom módszer stb.

The Alon-Jaeger-Tarsi conjecture \cite{Alon-Tarsi, Jaeger}, first proposed in 1981  is the following:

\begin{conjecture}[Alon-Jaeger-Tarsi]\label{conj-AJT}
For any field $\mathbb{F}$ with $|\mathbb{F}|\geq 4$ and any nonsingular matrix $M$ over $\mathbb{F}$, there is a vector $x$ such that both $x$ and $Mx$ have only nonzero entries.
\end{conjecture}

Alon and Tarsi \cite{Alon-Tarsi} reached this conjecture while trying to generalize some simple properties of sparse graphs to more general matroids. 
%%% Jaeger's motivation-t megkeresni
Note that the conjecture is trivial when $\mathbb{F}$ is an infinite field. Alon and Tarsi \cite{Alon-Tarsi} settled the case when $|\mathbb{F}|\geq 4$ is a {\it proper} prime power using the polynomial method from \cite{Alon}.

However, the case $\mathbb {F}=\mathbb{F}_p$ (with $p\geq 5$ being a prime) remained wide open. 

Motivated by this method for proving Conjecture~\ref{conj-AJT} Kahn proposed the so-called Permanent conjecture which states that the permanent rank of the matrix $(M M)$ is $n$ for every $n\times n$ nonsingular matrix $M$. (The permanent rank of a matrix is defined to be the   size  of  the  largest  square  submatrix  with   nonzero  permanent.) According to the polynomial argument of Alon and Tarsi \cite{Alon-Tarsi} the Permanent conjecture implies Conjecture~\ref{conj-AJT}.

For more results in the direction of proving the permanent conjecture see \cite{Kirkup}.

\noindent
For a discussion of the proof of Alon and Tarsi using Rédei-polynomials, see also \cite{Ball}.

DeVos formulated an even more general version of Conjecture~\ref{conj-AJT}, the so-called Choosability conjecture. We say that an $ m \times n $ matrix $ M $ is $ (a,b) $-choosable if for all subsets $ X_1,X_2,\dots,X_m \subseteq {\mathbb F} $ with $ |X_i|=a $ and $ Y_1,Y_2,\ldots,Y_n \subseteq {\mathbb F} $ with $ |Y_j|=b $, there exists a vector $ x \in X_1 \times X_2 \times \ldots \times X_m $ and a vector $ y \in Y_1 \times Y_2 \times \ldots \times Y_n $ so that $ M x=y $. DeVos \cite{DeVos} proved that every invertible matrix over a field of characteristic 2 is $ (k+1,|{\mathbb F}|-k) $-choosable for every $ k $, which implies that Conjecture~\ref{conj-AJT} is true in a very strong sense for these fields. The Choosability in $ {\mathbb Z}_p $ conjecture (DeVos)  asserts that invertible matrices over fields of prime order have choosability properties nearly as strong: Every invertible matrix with entries in $ {\mathbb Z}_p $ for a prime $ p $ is $ (k+2,p-k) $-choosable for every $ k $.

%%% Majd még Permanent conjecture stb.-ről írni
%%% additive bases, 3-flow? 

Despite the large interest on the problem only a few further partial results has been obtained after \cite{Alon-Tarsi}. Alon  has  shown  that  for $M$ chosen uniformly from the nonsingular $n\times n$ matrices over $\mathbb{F}$, Conjecture~\ref{conj-AJT} is true almost  surely  as $n\to \infty$. Yu \cite{Yu} proved that for $n<2^{p-2}$ Conjecture~\ref{conj-AJT} holds for matrices of size $n\times n$ by establishing this result for the Permanent conjecture. Akbari et al \cite{Akbari} proved that every nonsingular matrix is similar to a matrix which satisfies the statement of Conjecture~\ref{conj-AJT}. 

In this paper we prove that Conjecture~\ref{conj-AJT} holds, if $p$ is sufficiently large:

%%%% másik tétel 
\begin{theoremA} 
Let $p$ be a prime such that $61<p\ne 79$.  If $M$ is an $n\times n$ nonsingular matrix over $\mathbb{F}_p$ and   $c_i, d_i \in \mathbb{F}_p\ (1\leq i\leq n)$, then there exists a vector $x \in \mathbb{F}_p^n$ such that $x_i \neq c_i$ for each $1 \leq i \leq n$ and $(Mx)_i \neq d_i$ for each $1 \leq i \leq n$.
\end{theoremA}
Note that the choice $c_1=\dots=c_n=d_1=\dots=d_n=0$ gives that Conjecture~\ref{conj-AJT} holds for primes satisfying  $61<p\ne 79$.

Note that the theorem of Alon and Tarsi about the {\it proper} prime power case also has some nice applications about the hat guessing number of the complete bipartite graph $K_{n,n}$ in \cite{Alon-Ben}. Their proof together with Theorem A also cover the case of prime fields when $61<p\ne 79$.

We also prove a generalization where for each entry of $x$ and $Mx$ there can be $k$ forbidden values. In this case the statement also holds if $p$ is sufficiently large:
\begin{theoremB} 
There exists a universal constant $K$ such that the following statement holds. Let $k \geq 2$ be an integer and $p$ be a prime such that $p > k^{K k}$. For $1\leq i\leq n$ let $X_i,Y_i\subseteq \mathbb{F}_p^n$ be subsets of size at least $p-k$. Then there exists a vector $x$ such that 
$$x\in X_1\times X_2\times \ldots \times X_n\quad \text{and}\quad
Mx\in Y_1\times Y_2\times \ldots\times Y_n.$$
\end{theoremB}

%%%Thm B régebbi kimondása:
%\begin{theoremB} 
%There exists a universal constant $C$ such that the following statement holds. Let $k \geq 2$ be an integer and $p$ be a prime such that $p > k^{C k}$. Let $c_{i,\ell},d_{i,\ell}\in \mathbb{F}_p$ (for $1\leq i\leq n$, $1\leq \ell \leq k$). Then there exists a vector $x\in \mathbb{F}_p^n$ such that for every $1\leq i\leq n$ and $1\leq \ell \leq k$ we have $x_i\ne c_{i,\ell}$ and $(Ax)_i\ne d_{i,\ell}$. 
%\end{theoremB}

Our proofs make use of group ring identities and 
we propose the following conjecture:

\begin{conjecture}\label{41}
Let $p > 3$ be a prime and let $M$ be a nonsingular $n \times n$ matrix over $\mathbb{F}_p$ with
row vectors $a_1, \dots,a_n$, finally, let the standard basis vectors of $\mathbb{F}_p^n$ be $e_1,\dots,e_n$.

\noindent
Then the $\mathbb{Z}$ group ring identity 
$$\prod_{1 \leq i \leq n} (1 - g^{e_i}) \cdot   \prod_{1 \leq i \leq n} (1 - g^{a_i})  = 0 \in \mathbb{Z}[G]$$
follows from the mod $p$ group ring identity 
$$\prod_{1 \leq i \leq n} (1 - g^{e_i}) \cdot   \prod_{1 \leq i \leq n} (1 - g^{a_i})  = 0 \in \mathbb{F}_p[G].$$
\end{conjecture}

For the primes $p$ satisfying $61 < p \neq 79$ we will prove that even the seemingly weaker mod $p$ group ring identity cannot hold for any nonsingular matrix $M$.
On the other hand, if $p$ is {\it small} we could not prove Conjecture~\ref{41}, but we will show in Section~\ref{sec-small-primes} that the Alon-Jaeger-Tarsi conjecture follows from this conjecture in the case of small primes, as well.

\begin{theoremC}
For every prime $p>3$ Conjecture~\ref{41} implies the Alon-Jaeger-Tarsi conjecture.
\end{theoremC}

The organization of the paper is as follows. In Section~\ref{sec-not}  some necessary notations are introduced. %about group rings and polynomial functions on $\mathbb{F}_p^n$.
In Section~\ref{sec-outline}  we outline the strategy of the proof of Theorem A.  In Section~\ref{sec-prelim}  we prove some additive combinatorial lemmas about the minimal size of subsets in $\mathbb{F}_p$ containing arithmetic progressions around all elements. In Section~\ref{sec-group-ring} we investigate different properties of matrices $M \in M^{n \times n}(\mathbb{F}_p)$ which are {\it similar} to being a counterexample to the Alon-Jaeger-Tarsi conjecture and we prove different equivalent characterisations of them. 

Section~\ref{sec-proof}  contains the proofs of  Theorem A and Theorem B, and Section~\ref{sec-small-primes} contains the proof of Theorem C.
%\bigskip
%In Section 7) we prove Theorem \textbf{C}, so that Conjecture~\ref{41} implies the Alon-Jaeger-Tarsi conjecture
%for all primes $p > 3$.
%\bigskip
 Some open questions and conjectures are formulated in Section~\ref{sec-open}. 
Finally, we include an Appendix containing explicit constructions of subsets in $\mathbb{F}_p$ for different primes, when there are  constructions that are good enough for our purposes but our general additive combinatorial lemmas do not cover these cases.

%%%%%%%%%%%%%%
%%%%%%%%%%%%%%
%%%%%%%%%%%%%%

\section{Notations}\label{sec-not}

Throughout the paper we use the notation $[n]:=\{1,2,\dots,n\}$ and we denote by $[a,b]:=\{k\ |\ a\leq k\leq b,\ k\in\mathbb{Z} \}$  the set of integers between $a$ and $b$ included.

The set of $n\times n$ matrices over a field $\mathbb{F}$ is denoted $M^{n\times n}(\mathbb{F})$. For a matrix $M\in M^{n\times n}(\mathbb{F})$ the $(n-1)\times (n-1)$ submatrix obtained by deleting the $i$-th row and the $j$-th column from $M$ is denoted $M_{ij}$.

The standard dot product of vectors $v,w$ from an $n$-dimensional vector space is denoted $\langle v, w \rangle := \sum\limits_{1 \leq i \leq n} v_i w_i$.

We use the notation $\mathbb{F}_p^*:=\mathbb{F}_p\setminus\{0\}$ for the set of nonzero elements of the prime field $\mathbb{F}_p$.

In the paper if we consider a group ring $R[G]$, where $R$ is an arbitrary ring and $G = \mathbb{F}_p^n$ for some $n \geq 1$ and $v \in \mathbb{F}_p^n$ is an arbitrary vector, then we use the multiplicative notation $g^v = \prod\limits_{1 \leq i \leq n} g_i^{v_i}$.

%%%%
%%%%

If we have a polynomial $f(x_1, \dots, x_n) \in \mathbb{F}_p[x_1, \dots, x_n]$, by its reduced form $$\overline{f}(x_1, \dots, x_n) \in \mathbb{F}_p[x_1, \dots, x_n]$$
we mean the polynomial which has degree at most $p-1$ in each variable and obtained by the identification of $x^{i(p-1) + j} \cong x^j$ if $i>0$ and $ p>j > 0$. Note that $f(a_1, \dots, a_n) = \overline{f}(a_1, \dots, a_n)$ for every $a_1, \dots, a_n \in \mathbb{F}_p$.

Let us recall the simple fact that for a polynomial $f(x_1, \dots, x_n) \in \mathbb{F}_p[x_1, \dots, x_n]$ the following are equivalent:
\begin{itemize}
    \item $f(a_1, \dots, a_n) = 0$ for every $a_1, \dots, a_n \in \mathbb{F}_p$,
    \item $\overline{f}(a_1, \dots, a_n) = 0$ for every $a_1, \dots, a_n \in \mathbb{F}_p$,
    \item $\overline{f} \equiv 0$.
\end{itemize}

For $f \in \mathbb{F}_p[x_1, \dots, x_n]$ and $a_1,\dots,a_n\geq 0$ let 
$\Coeff\left(\prod\limits_{i=1}^n x_i^{a_i},f\right)$
denote the coefficient of the monomial $\prod\limits_{i=1}^n x_i^{a_i}$ in $f$.

We will use the notation
$$\sigma_k(x_1,\dots,x_n)=\sum\limits_{1\leq i_1<\dots<i_k\leq n} x_{i_1}\dots x_{i_k}$$
for the elementary symmetric polynomials.

%  we have $f(a_1, \dots, a_n) = 0$ for every $a_1, \dots, a_n \in \mathbb{F}_p$ if and only if  $\overline{f}(a_1, \dots, a_n) = 0$ for every $a_1, \dots, a_n \in \mathbb{F}_p$ if and only if
%$\overline{f} \equiv 0$.

%%%%
%%%%

\section{Outline of the proof}\label{sec-outline}

In this section the strategy of the proof is described. For brevity it is explained in case of the proof of  Theorem A, that is, the Alon-Jaeger-Tarsi conjecture (for sufficiently large primes).

In Section~\ref{sec-group-ring} we will investigate different properties of matrices $M \in M^{n \times n}(\mathbb{F}_p)$ which are {\it similar} to being a counterexample to the Alon-Jaeger-Tarsi conjecture and we will prove different equivalent characterisations of them. Specially, it is going to be shown that if $M \in M^{n \times n}(\mathbb{F}_p)$ is a counterexample to the Alon-Jaeger-Tarsi conjecture, then we have
\begin{equation}\label{firstZgroup}
\prod\limits_{1 \leq i \leq n} (1 - g^{e_i}) \cdot \prod\limits_{1 \leq i \leq n} (1 - g^{a_i}) = 0 \in \mathbb{Z}[\mathbb{F}_p^n],  
\end{equation}
where $a_i\ (i\in[ n])$ denotes the $i$-th row vector of the matrix $M$ and $e_i\ (i \in [n])$ denotes the $i$-th standard basis vector.

We will also investigate the following more special identity in the mod $p$ group ring:
\begin{equation}\label{firstpgroup}
\prod\limits_{1 \leq i \leq n} (1 - g^{e_i}) \cdot \prod\limits_{1 \leq i \leq n} (1 - g^{a_i}) = 0 \in \mathbb{F}_p[\mathbb{F}_p^n],  
\end{equation}

It will turn out that both identity~ \eqref{firstZgroup} and identity~\eqref{firstpgroup} can be formulated in equivalent forms as the vanishing of different coefficients in some polynomial functions. Moreover, these vanishings are partially in some kind of duality for the two identities (see Lemma~\ref{duality}). We will prove that even the seemingly weaker identity~\eqref{firstpgroup} cannot hold for an invertible matrix $M \in M^{n \times n}(\mathbb{F}_p)$ if $p > 61,\ p \neq 79$.

We will use the philosophy of the duality described above in a combinatorial way.
We are using the falseness of a statement similar to DeVos' conjecture in the dual setup in an extremal case (which does not satisfy the conditions of DeVos' conjecture), when it is false tautologically by a simple counting argument.

Let us call a subset $A$ $S_1$-type if for every element $a \in A$ there exist two different elements $b, c \in A$ such that $2a = b+c$. We will construct $S_1$-type subsets $A_i,B_i \subseteq \mathbb{F}_p^*\ (i \in [n])$  such that if $y_i \in A_i\ (i \in [n])$ and $x_i \in B_i \ (i \in [n])$, then 
$$\sum_{1 \leq i \leq n} y_i e_i \neq \sum_{1 \leq i \leq n} x_i a_i.$$

This is indeed the falseness of a DeVos-type statement, since, in other words the hyperplanes in $\mathbb{F}_p^n$ given by 
$y_i \in \mathbb{F}_p \setminus A_i\ (i \in [n])$, $x_i \in \mathbb{F}_p \setminus B_i\  (i\in[ n])$ cover
$\mathbb{F}_p^n$.

This is the point where we will use the assumptions $p > 61,\ p \neq 79$ as we can construct the above-mentioned subsets $A_i,B_i$ only in these cases.

We will use the existence of the subsets $A_i,B_i \subseteq \mathbb{F}_p^*\ (i\in [n])$ in a tricky way to show that if the mod $p$ group ring identity~\eqref{firstpgroup} holds, then the mod $p$ group ring identity should also hold for a particular $(n-1) \times (n-1)$ nonsingular submatrix of the original matrix $M$. We will reach a contradiction by induction and this will finish the proof of Theorem A completely.

\section{Preliminary additive combinatorial lemmas}\label{sec-prelim}
In this section we prove some additive combinatorial lemmas that we are going to use in the proofs of our main theorems.

\subsection{Subsets of $\mathbb{F}_p$ with many arithmetic progressions}

For integers $k \geq 1$ and primes $p\geq 2k+1$ let us define the following quantities.

Let us call a subset $A \subseteq \mathbb{F}_p$ {\it $S_k$-type} if for every $a \in A$ there exists a residue $d \in \mathbb{F}_p^*$ such that $a + i \cdot d \in A$ for every $-k \leq i \leq k$. In other words, a subset $A$ is $S_k$-type if each element of $A$ is the middle point of a $(2k+1)$-term arithmetic progression contained in $A$. Let $S_k(p) \geq 2k + 1$ denote the size of the smallest $S_k$-type subset $A \subseteq \mathbb{F}_p$.

%Let us denote by $S_k(p) \geq 2k + 1$ the minimal integer such that there exists a subset $A \subseteq \mathbb{F}_p$ with 
%$|A| = S_k(p)$ such that $A$ is $S_k$-type.

\bigskip

Let us call a subset $A \subseteq \mathbb{F}_p$ {\it $N_k$-type} if the following two conditions hold:
\begin{itemize}
    \item For every $a \in A$ there exists a residue $d \in \mathbb{F}_p^*$ such that $a + i \cdot d \in A$ for every $-k \leq i \leq k$.
    \item For every $a \notin A$ there exists a residue $d \in \mathbb{F}_p^*$ such that $a + i \cdot d \in A$ for every $1 \leq i \leq k$.
\end{itemize}
That is, a subset $A$ is $N_k$-type, if $A$ is $S_k$-type and for every $b\notin A$ the set $A\cup \{b\}$ contains a $(k+1)$-term arithmetic progression with starting element $b$. 
%Let us denote by $N_k(p) \geq 2k + 1$ be the minimal integer such that there exists a subset $A \subseteq \mathbb{F}_p$ with 
%$|A| = N_k(p)$ such that $A$ is $N_k$-type.
Let $N_k(p) \geq 2k + 1$ denote the size of the smallest $N_k$-type subset $A \subseteq \mathbb{F}_p$.

\smallskip

\noindent
Note that  $S_1(p) = N_1(p)$ trivially holds and in general we have $S_k(p) \leq N_k(p)$.

\noindent
The following lemmas about the quantities $S_k(p), N_k(p)$ will be applied in the proofs.

\begin{lemma}\label{SN1}
If $p \geq 3$, then there exists a subset 
$A \subseteq \mathbb{F}_p$ such that 
$$|A| \leq 2\lfloor \log_2 p\rfloor$$
and for every $a \in A$ there exist $b, c \in A$ such that $2a = b + c$ and $b \neq a$, that is, $A$ is $S_1$-type. In other words 
$$S_1(p) = N_1(p) \leq 2\lfloor \log_2 p\rfloor.$$
\end{lemma}
%%%%

%%%% Ez kisebb, mint sqrt(p-1) ha (197<p ÉS p\ne 257)

\begin{proof}
We will construct a set $A$ of size $2\lfloor \log_2 p\rfloor$ which is $S_1$-type. Let $\{-1,0,1,(p-1)/2\}\subseteq A$, then
$$0=1+(-1),\quad p-1=(-1)+0$$
show that for $a=0$ and $a=(p-1)/2$ it is possible to find $b,c\in A$ in such a way that $2a=b+c$. We will add further elements to the set in such a way that the condition will also hold for $a=\pm 1$ and for the new elements, as well. Each new element will be either $t/2$ or $(t-1)/2$ or $(t+1)/2$, where $t$ is an element already contained in $A$. The representations 
$$2\cdot \frac{t}{2}=t+0,\quad 2\cdot\frac{t-1}{2}=t+(-1),\quad 2\cdot\frac{t+1}{2}=t+1$$
show that the condition will hold for the newly added elements. This also implies that if 1 and $-1$ can also be expressed in the form $t/2$ or $(t-1)/2$ or $(t+1)/2$ with some $t\in A \setminus \{-1, 1\}$, then we obtained an $S_1$-type set.

Let us distinguish the cases $p\equiv 1\pmod {4}$ and $p\equiv 3\pmod {4}$. In the former case let $p=4k+1$, then $(p-1)/2=2k$, and we may add $(2k)/2=k$ and $(2k+1)/2\equiv 3k+1\equiv -k\pmod{p}$ to the set $A.$ Now, our aim is to reach 1 from $k$: we add $\lfloor k/2\rfloor$ to $A$ (which is  $k/2$ if $k$ is even and  $(k-1)/2$ if $k$ is odd). Then we repeat the process with the newly added element. This way in $\lfloor \log_2 k\rfloor$ steps we will reach 1 from $k$:
$$\underbrace{k, \lfloor k/2\rfloor ,\lfloor k/4\rfloor,\dots,1.}_{\lfloor \log_2 k\rfloor+1\text{ elements}}$$
Similarly, we can reach $-1$ with the sequence
$$\underbrace{-k, -\lfloor k/2\rfloor ,-\lfloor k/4\rfloor,\dots,-1.}_{\lfloor \log_2 k\rfloor+1\text{ elements}}$$
Thus we obtain an $S_1$-type set of size $2(\lfloor \log_2 k\rfloor+1)+2=2\lfloor \log_2 p\rfloor$, since $k=(p-1)/4$.

The case $p\equiv 3\pmod{4}$ can be handled similarly. Let $p=4k+3$, then $(p-1)/2=2k+1$, so we may add $((2k+1)+1)/2=k+1$ and $(2k+1)/2\equiv 3k+2\equiv -(k+1)\pmod{p}$ to the set $A$. Then, by following the previous process we get an $S_1$-type set of size  
$2(\lfloor \log_2 (k+1)\rfloor+1)+2=2\lfloor \log_2 p\rfloor$, since $k+1=(p+1)/4$.

\end{proof}

\begin{corollary}\label{cor-s1}
If $61<p\ne 79$ is a prime, then $S_1(p)^2<p-1$.
\end{corollary}

\begin{proof}
By Lemma~\ref{SN1} we obtain that $S_1(p)^2\leq  4(\lfloor \log_2 p\rfloor)^2<p-1$, if $199\leq  p \neq 257$.

If $61 < p < 199,\ p \neq 79$ or $p = 257$, then from the constructions described in the Appendix we get again that
$S_1(p)^2 < p-1$.

\end{proof}

\noindent
For the case $k>1$ we prove a weaker statement:

\begin{lemma}\label{Skepsilon}
Let $k \geq 2$ be an integer. Then there exists a universal constant $K \geq 32$ such that if $p> k^{K k}$ is a prime, then there exists a subset $A \subseteq \mathbb{F}_p$ such that $|A| \leq p^{\frac{1}{2k+2}}$ and $A$ is $S_k$-type. In other words, for $p>k^{Kk}$ we have 
$$S_k(p)\leq p^{\frac{1}{2k+2}}.$$
\end{lemma}
\begin{proof}

We will give an explicit construction. Let $x: = \lceil p^{\frac{1}{4k + 4}} \rceil $,
that is, we have $p \leq x^{4k +4} \leq p\cdot 2^{4k+4}\leq p \cdot 2^{6k}$.

For a residue $c \in \mathbb{F}_p$, a subset $C \subseteq \mathbb{F}_p$ and a residue $d \in \mathbb{F}_p^*$, let $i$ be the minimal positive integer such that $c + i \cdot d \in C$ and let us denote $c + i \cdot d$ by $f(c, C, d)$.

Similarly for a residue $c \in \mathbb{F}_p$, a subset $C \subseteq \mathbb{F}_p$ and a residue $d \in \mathbb{F}_p^*$, let $i$ be the maximal negative integer such that $c + i \cdot d \in C$ and let us denote $c + i \cdot d$ by $g(c, C, d)$.

We will define subsets $A_i, B_i, C_i \ (i \in[ 4k+3])$ and $A_{4k+4}$ in such a way that $A := \bigcup\limits_{1 \leq i \leq 4k + 4} A_i$ will satisfy the requirements.

Let 
$$A_1: = \{ kx, kx\},$$
$$B_1 := [-kx, -kx + k] \cup \{ f(j, A_1, x)\ |\ kx - k \leq j \leq kx \}$$
and
$$C_1 = [kx-k , kx] \cup \{ g(j, A_1, x)\ |\ -kx \leq j \leq -kx + k \}.$$

Next, let us define 
$$A_2 := \{a + j \cdot x\ |\ a \in C_1,\ 1 \leq j \leq 2kx\} \cup \{a - j \cdot x\ |\ a \in B_1,\ 1 \leq j \leq 2kx\},$$
\begin{multline*}
B_2 := \{a - j \cdot x\ |\ a \in B_1,\ 2kx - k \leq j \leq 2kx\} \cup \\
\cup\{f(a + j \cdot x, A_2, x^2)\ |\ a \in C_1,\ 2kx - k \leq j \leq 2kx\}    
\end{multline*}
and 
\begin{multline*}
    C_2: = \{a + j \cdot x\ |\ a \in C_1,\ 2kx - k \leq j \leq 2kx\} \cup \\
    \cup \{g(a - j \cdot x, A_2, x^2)\ |\ a \in B_1,\ 2kx - k \leq j \leq 2kx\}.
\end{multline*}

\noindent
The sets $A_i,B_i,C_i$ are defined iteratively in a similar fashion. Namely, in general, for $2 \leq i \leq 4k + 3$ we define 
$$A_i := \{a + j \cdot x^{i-1}\ |\ a \in C_{i-1},\ 1 \leq j \leq 2kx) \cup \{a - j \cdot x^{i-1}\ |\ a \in B_{i-1}, 1 \leq j \leq 2kx\},$$
\begin{multline*}
    B_i := \{a - j \cdot x^{i-1}\ |\ a \in B_{i-1},\ 2kx - k \leq j \leq 2kx\} \cup \\
    \cup\{f(a + j \cdot x^{i-1}, A_i, x^i)\ |\ a \in C_{i-1},\ 2kx - k \leq j \leq 2kx\}
\end{multline*}
and 
\begin{multline*}
    C_i: = \{a + j \cdot x^{i-1}\ |\ a \in B_{i-1},\ 2kx - k \leq j \leq 2kx\} \cup \\
    \cup \{f(a + j \cdot x^{i-1}, A_i, x^i\ |\ a \in C_{i-1},\ 2kx - k \leq j \leq 2kx\}.
\end{multline*}

Finally, let us set 
$$A_{4k + 4} := \bigcup\limits_{a \in C_{4k + 3}}\{a, a + x^{4k + 3}, \dots, f(a, A_{4k + 3}, x^{4k + 3})\}.$$

One can see easily from the construction that there is a universal constant $K \geq 32$ such that if $p> k^{K k}$ is a prime large enough, then the subset $A$ is $S_k$-type and $|A| \leq p^{\frac{1}{2k+2}}$.
\end{proof}

\begin{lemma}\label{Nkepsilon}
Let $k \geq 2$ be an integer. If $p> k^{32k}$ is a prime, then $\mathbb{F}_p$ can be expressed as a disjoint union $A_1 \cup \dots \cup A_{\lceil p^{1/(2k+1)} \rceil}$ in such a way that for every $1 \leq j \leq \lceil p^{1/(2k+1)} \rceil$ the set $A_j$ is $N_k$-type.

In particular, for $p> k^{32 k}$ we have 
$$N_k(p) \leq \frac{p}{\lceil p^{1/(2k+1)} \rceil} \leq p^{2k/(2k+1)} .$$
\end{lemma}

\begin{proof}

The sets $A_1, A_2, \dots , A_{\lceil p^{1/(2k+1)} \rceil}$ are constructed by a random construction.

Let the probability of $a \in A_i$ be $\frac{1}{\lceil p^{1/(2k+1)} \rceil}$ for every index $1\leq i\leq \lceil p^{1/(2k+1)} \rceil$ (independently for every element $a$). 
%Independently for every element $a \in \mathbb{F}_p$, for an index $1 \leq i \leq \lceil p^{1/(2k+1)} \rceil$ let the probability of $a \in A_i$ be $\frac{1}{\lceil p^{1/(2k+1)} \rceil}$.
We will prove that the disjoint union $\mathbb{F}_p = A_1 \cup \dots \cup A_{\lceil p^{1/(2k+1)} \rceil}$ is fine with positive probability.

Observe that for an index $j\leq \lceil p^{1/(2k+1)} \rceil$ and an element $a \in A_j \subseteq \mathbb{F}_p$ the probability that there does not exist $d \in \mathbb{F}_p^*$ such that $a + i \cdot d \in A$ for every $-k \leq i \leq k$ is at most 
$$\left(1- \frac{1}{(\lceil p^{1/(2k+1)} \rceil)^{2k}} \right)^{(p-1)/(2k^2)}.$$

Indeed, residues $d_1, \ldots, d_{(p-1)/(2k^2)}$ can be chosen in such a way that the sets $\{a + i \cdot d_j:\  -k \leq i \leq k,\ i \neq 0\}$ are pairwise disjoint and then we can estimate the probability using independence. (The $d_j$ values can be chosen greedily: after choosing $d_j$ we can not chose any element of the form $\pm d_juv^{-1}$ with $u,v\in [k]$, so choosing one $d_j$ forbids at most $2k^2$ elements to be chosen.)

Similarly, if $a \notin A_j \subseteq \mathbb{F}_p$, the probability that there does not exist $d \in \mathbb{F}_p^*$ such that $a + i \cdot d \in A$ for every $ 1 \leq i \leq k$ is at most 
$$\left(1- \frac{1}{(\lceil p^{1/(2k+1)} \rceil)^{k}} \right)^{(p-1)/(k^2)}.$$

Therefore, the probability that the disjoint union $\mathbb{F}_p = A_1 \cup \dots \cup A_{\lceil p^{1/(2k+1)} \rceil}$ is not fine is less than 
$$p \cdot \lceil p^{1/(2k+1)} \rceil \cdot \left(1- \frac{1}{(\lceil p^{1/(2k+1)} \rceil)^{2k}} \right)^{(p-1)/(2k^2)},$$
which is less than $1$, if $p$ is large enough.

Indeed, by using the  estimation $1-x<e^{-x}$ we get that
$$p \cdot \lceil p^{1/(2k+1)} \rceil \cdot \left(1- \frac{1}{(\lceil p^{1/(2k+1)} \rceil)^{2k}} \right)^{(p-1)/(2k^2)}<p^2\cdot e^{- \frac{p-1}{2k^2 \cdot (\lceil p^{1/(2k+1)} \rceil)^{2k} }}<1,$$
%have $\left(1- \frac{1}{(\lceil p^{1/(2k+1)} \rceil)^{2k}} \right)^{(\lceil p^{1/(2k+1)} \rceil)^{2k}} < \frac{1}{e}$, so we have:
%\begin{equation*}
%p \cdot \lceil p^{1/(2k+1)} \rceil \cdot \left(1- \frac{1}{(\lceil p^{1/(2k+1)} \rceil)^{2k}} \right)^{(p-1)/(4k^2)} 
%< p^2 \cdot e^{- \frac{p-1}{4k^2 \cdot (\lceil p^{1/(2k+1)} \rceil)^{2k} }}.
%\end{equation*}
%It means that the probability is less than $1$ 
if $\frac{p-1}{2k^2 \cdot (\lceil p^{1/(2k+1)} \rceil)^{2k} } > 2 \log p$, 
which  holds if $p > k^{32k}$.% for some universal constant $C$.

The second statement holds trivially, since, if $A_1 \cup \dots \cup A_{\lceil p^{1/(2k+1)} \rceil} = \mathbb{F}_p$ for pairwise disjoint sets $A_j$, then there is an index $1 \leq j \leq \lceil p^{1/(2k+1)} \rceil$ such that $|A_j| \leq \frac{p}{\lceil p^{1/(2k+1)} \rceil} \leq p^{2k/(2k+1)}$.
\end{proof}

\smallskip

\noindent
The following folklore lemma will also be needed:

\begin{lemma}\label{randomdisjoint}
Let $n \geq 1$ be an integer and $p$ be a prime.
Let $\{e_1, \dots, e_n\}$ and $\{f_1, \dots, f_n\} $ be two bases in $\mathbb{F}_p^n$.
Assume that we have subsets $U_i,V_i \subseteq \mathbb{F}_p^*\  (i \in  [n])$ such that 
$$\prod_{1 \leq i \leq n} |U_i| \cdot \prod_{1 \leq i \leq n} |V_i| < (p-1)^n.$$
Then there exist residues $\lambda_i \in \mathbb{F}_p^*\ (i\in [n])$ such that if $x_i \in U_i,\ y_i\in V_i\ (i \in[ n])$, then 
$$\sum_{1 \leq i \leq n} x_i e_i \neq \sum_{1 \leq i \leq n} \lambda_i y_i f_i.$$
\end{lemma}

\begin{proof}

Let us choose the multipliers $\lambda_i \in \mathbb{F}_p^*$  independently at random taking each residue in $\mathbb{F}_p^*$ with probability $\frac{1}{p-1}$.

For fixed elements $x_i \in U_i\ (i\in[ n])$ and $y_i \in V_i\ (i\in[ n])$ the probability that the equation 
$$\sum_{1 \leq i \leq n} x_i e_i = \sum_{1 \leq i \leq n} \lambda_i y_i f_i$$
holds is at most $\frac{1}{(p-1)^n}$.

Therefore, the probability that the multipliers $\lambda_i \in \mathbb{F}_p^*$ do not satisfy the requirement of the lemma is at most $$\frac{\prod\limits_{1 \leq i \leq n} |U_i| \cdot \prod\limits_{1 \leq i \leq n} |V_i|}{(p-1)^n} < 1,$$
which finishes the proof.

\end{proof}

\noindent
We immediately get the following corollary of Lemma~\ref{randomdisjoint}:

\begin{corollary}\label{randomdisjoint2}
Let $n, k \geq 1$ be integers  and $p$ be a prime such that $N_k(p) \cdot S_k(p) < p-1$.
Let $\{e_1, \dots, e_n\}$ and $\{f_1, \dots, f_n\} $ be two bases in $\mathbb{F}_p^n$. Then there exist 
subsets $A_i, B_i \subseteq \mathbb{F}_p^*\ ( i \in[ n])$  such that $A_i$ is $S_k$-type for every $i\in[ n]$, $B_i$ is $N_k$-type for every  $i\in [n]$ and if $x_i \in A_i,\ y_i \in B_i$ (for $1\leq i\leq n)$, then 
$$\sum_{1 \leq i \leq n} x_i e_i \neq \sum_{1 \leq i \leq n} y_i f_i.$$
\end{corollary}

\begin{proof}

Since the $S_k$ and $N_k$ properties are stable to shifting and dilation, we may assume that $A \subseteq \mathbb{F}_p^*$ is $S_k$-type with $|A| = S_k(p)$ and $B \subseteq \mathbb{F}_p^*$ is $N_k$-type with $|B| = N_k(p)$.

Now, let us apply Lemma~\ref{randomdisjoint} for the choice $U_1=\dots=U_n=A$ and $V_1=\dots=V_n=B$. Then Lemma~\ref{randomdisjoint} implies that there exist residues $\lambda_i \in \mathbb{F}_p^*$ such that if $x_i \in U_i,\ y_i\in V_i$ (for $1 \leq i \leq n)$, then 
$$\sum_{1 \leq i \leq n} x_i e_i \neq \sum_{1 \leq i \leq n} \lambda_i y_i f_i.$$

The subsets $A_i = U_i, \ B_i = \lambda_i V_i \ (i\in [ n])$ satisfy the conditions of the corollary.
\end{proof}

\noindent
From the preceding lemmas we get the following corollary which we will use in the proof of our main theorems.

\begin{corollary}\label{goodsubsets}
Let $n, k \geq 1$ be integers, $p$ be a prime, and  let $\{e_1, \dots, e_n\}$ and  $\{f_1, \dots, f_n\}$ be  two bases in $\mathbb{F}_p^n$. Then the following statements hold:

\begin{itemize}
\item[(i)] If $k = 1 $ and $61 < p \neq 79$, then there exist subsets $A_i,B_i \subseteq \mathbb{F}_p^*\ (i\in [n])$  such that $A_i$ and $B_i$ are $S_1$-type (for every  $i\in [n]$) and if $x_i \in A_i,\ y_i\in B_i\ (i\in[ n])$, then 
$$\sum_{1 \leq i \leq n} x_i e_i \neq \sum_{1 \leq i \leq n} y_i f_i.$$

\item [(ii)] If $k \geq 2$ and $p > k^{K k}$ (where $K$ is the universal constant from Lemma~\ref{Skepsilon}), then there exist subsets $A_i,B_i \subseteq \mathbb{F}_p^*\ (i \in [n])$  such that $A_i$ is $S_k$-type for every $i\in [n]$, $B_i$ is $N_k$-type for every $i\in [n]$ and if $x_i \in A_i,\ y_i\in B_i \ (i\in [n])$, then 
$$\sum_{1 \leq i \leq n} x_i e_i \neq \sum_{1 \leq i \leq n} y_i f_i.$$

\end{itemize}

\end{corollary}

\begin{proof}

Part (i) follows from Corollary~\ref{cor-s1} and Corollary~\ref{randomdisjoint2}. %if $p-1 > 4(\lfloor \log_2 p\rfloor)^2$, which holds
%if $199\leq  p \neq 257$.
If $61 < p < 199,\ p \neq 79$ or $p = 257$, then from the constructions described in the Appendix we get again that
$S_1(p)^2 < p-1$, so the statement follows again from Corollary~\ref{randomdisjoint2}.

Part (ii) follows from Lemma~\ref{Skepsilon},  Lemma~\ref{Nkepsilon} and Corollary~\ref{randomdisjoint2} since 
$$N_k(p) \cdot S_k(p)\leq p^{1-\frac{1}{(2k+1)(2k+2)}} <p-1$$ if $k \geq 2$ and $p > k^{K k}$.

\end{proof}

\section{Different properties of counterexamples and their connection}\label{sec-group-ring}

In this section we introduce several kind of properties of counterexamples and prove results about their connection.

\noindent 
Let us prove first the following easy lemma.
 \begin{lemma}
Let $\{e_1,\dots,e_n\}$ be a basis for $\mathbb{F}_p^n$ and let $w_i := 1 - g^{e_i} \in \mathbb{F}_p[\mathbb{F}_p^n]$. Then the elements $\prod\limits_{1 \leq i \leq n} w_i^{b_i}$  
 $(0 \leq b_i \leq p-1,\ i \in [n])$ provide a basis for $\mathbb{F}_p[\mathbb{F}_p^n]$ as a vector space.
\end{lemma}

\begin{proof}
It is clear that the elements $\prod\limits_{1 \leq i \leq n} w_i^{b_i}$ with arbitrary exponents generate $\mathbb{F}_p[\mathbb{F}_p^n]$ as a vector space. On the other hand we know by calculation that $w_i^{p} = 0$ for each $1 \leq i \leq n$, thus we obtain that the elements $\prod\limits_{1 \leq i \leq n} w_i^{b_i}$  $(0 \leq b_i \leq p-1,\ i\in[n])$ generate $\mathbb{F}_p[\mathbb{F}_p^n]$ as a vector space.

However, $\dim(\mathbb{F}_p[\mathbb{F}_p^n]) = p^n$, therefore, they indeed form a basis.
\end{proof}

Let us prove the following proposition about different properties which are {\it  ``similar''} to being a counterexample to the Alon-Jaeger-Tarsi conjecture. (For a proof of the equivalence of (P1) and (P3) in a special case see \cite{Szegedy}.)

\begin{proposition}\label{properties}
Let $p> 3$ be a prime, $n \geq 1$ be an integer and let $M$ be a nonsingular $n\times n$ matrix over $\mathbb{F}_p$. Let $\omega$ denote a $p$-th root of unity. For every $i\in [n]$ let  $0 \leq t_{i} \leq p-1$ and $0 \leq t'_{i} \leq p-1$, and
for each $i\in [n]$ let $t_i$ different residues $c_{i,k} \in \mathbb{F}_p\ (k\in [t_i])$ be  and $t'_i$ different residues $d_{i,k} \in \mathbb{F}_p\ (k\in [t'_i])$ are given.

Let us denote the row vectors of $M$  by $a_1, \dots, a_n$ and let the standard basis vectors be denoted by $e_1, \dots, e_n$. Then the following three statements are equivalent:

\begin{itemize}
\item[(P1)]  There does not exist a vector $x \in \mathbb{F}_p^n$ such that $x_i \neq c_{i, k}$ for each $1 \leq i \leq n,\ 1 \leq k \leq t_i$  and $(Mx)_i \neq d_{i, k}$ for each $1 \leq i \leq n,\ 1 \leq k \leq t'_i$.

\item[(P2)] The polynomial $$h(x_1, \dots, x_n) = \prod\limits_{\substack{1 \leq i \leq n,\\ 1 \leq k \leq t_i'}} \left( \sum\limits_{1 \leq j \leq n} a_{i, j} x_j - d_{i, k}  \right) \cdot \prod\limits_{\substack{1 \leq i \leq n,\\ 1 \leq k \leq t_i} } (x_i - c_{i, k})$$ vanishes for all entries $x_1 = z_1, \dots, x_n = z_n$ where $z_1, \dots, z_n \in \mathbb{F}_p$.

\item[(P3)] The element 
$$ \prod\limits_{\substack{1 \leq i \leq n,\\ 1 \leq k \leq t_i}} (1 - \omega^{-c_{i, k}} g^{e_i}) \cdot \prod\limits_{\substack{1 \leq i \leq n,\\ 1 \leq k \leq t'_i}} (1 - \omega^{-d_{i, k}} g^{a_i}) \in \mathbb{C}[\mathbb{F}_p^n]$$ is zero.

\end{itemize}

\smallskip

\noindent
The following statement follows from the equivalent statements \emph{(P1)-(P2)-(P3)}:

\begin{itemize}
\item[(P4)] The element $$\prod\limits_{1 \leq i \leq n} (1 - g^{e_i})^{t_i} \cdot \prod\limits_{1 \leq i \leq n} (1 - g^{a_i})^{t'_i} \in \mathbb{F}_p[\mathbb{F}_p^n]$$
is zero.
\end{itemize}

\smallskip
\noindent
The following statement follows from statement \emph{(P4)}:

\begin{itemize}
\item[(P5)] For the polynomial $f$ defined as
$$f(x_1, \dots, x_n) = \prod_{1 \leq i \leq n} \left( \sum_{1 \leq j \leq n} a_{i, j} x_j \right)^{t_i'} \cdot \prod_{1 \leq i \leq n} x_i^{t_i}$$
we have $\deg(\overline{f}) < \sum\limits_{1 \leq i \leq n} t_i + \sum\limits_{1 \leq i \leq n} t'_i$.

\end{itemize}
\noindent
That is, we have $(\emph{P}1)\leftrightarrow (\emph{P}2) \leftrightarrow (\emph{P}3) \to (\emph{P}4) \to (\emph{P}5)$.
\end{proposition}

\begin{proof}

The equivalence of the statements (P1) and (P2) is clear.

Let us prove  first the implication $\text{(P1)} \to \text{(P3)}$, so under the conditions of statement (P1) we have to prove that $$\prod\limits_{\substack{1 \leq i \leq n, \\1 \leq k \leq t_i}} (1 - \omega^{-c_{i, k}} g^{e_i}) \cdot \prod\limits_{\substack{1 \leq i \leq n,\\ 1 \leq k \leq t'_i}} (1 - \omega^{-d_{i, k}} g^{a_i}) = 0 \in \mathbb{C}[\mathbb{F}_p^n].$$

Let $C(\mathbb{F}_p^m)$ denote the space of complex valued functions $\mathbb{F}_p^m \to \mathbb{C}$. Let the usual Fourier transform $F : C(\mathbb{F}_p^n) \to C(\mathbb{F}_p^n)$ be given by
$$F(f) ( v) = \sum_{w \in \mathbb{F}_p^n}  f(w) \cdot \omega^{\langle v, w \rangle},$$
where $\langle v, w \rangle = \sum\limits_{1 \leq i \leq n} v_i w_i$ is the standard scalar product.

%Let's have the complex valued functions $\mathbb{F}_p^m \to \mathbb{C}$, let's denote their space by $C(\mathbb{F}_p^m)$, let's have a $p$-th root of unity $\omega$ and let's have the usual Fourier transform $F : C(\mathbb{F}_p^n) \to C(\mathbb{F}_p^n)$ given by
%$F(f) ( v) = \sum_{w \in \mathbb{F}_p^n}  f(w) \cdot \omega^{\langle v, w \rangle}$, where $\langle v, w \rangle = \sum\limits_{1 \leq i \leq n} v_i w_i$ is the standard scalar product.

It is well known that $F(F(f)) = R(f) * p^n$, where $R : C(\mathbb{F}_p^n) \to C(\mathbb{F}_p^n)$ is given by $R(f(x)) = f(-x)$. Also, for the convolution operation $* : C(\mathbb{F}_p^n) \otimes C(\mathbb{F}_p^n) \to C(\mathbb{F}_p^n)$ defined by 
$$f * g(v) = \sum_{w \in \mathbb{F}_p^n}f(w) \cdot g(v-w)  \cdot \frac{1}{p^n},$$
it is well known that $F(f * g) = F(f) \cdot F(g)$ and $F(f \cdot g) = F(f) * F(g)$.

Notice that we have a natural map $i : \mathbb{C}[\mathbb{F}_p^n] \to  C(\mathbb{F}_p^n)$ sending an element $\sum\limits_{v \in \mathbb{F}_p^n} c_v g^v$ to the function $f \in  C(\mathbb{F}_p^n)$ given by $f(v) = c_v$. Note that $i(x \cdot y) = i(x) * i(y) \cdot p^n$.

By using these facts we have to prove that 
\begin{equation}\label{eq-iomega} \prod\limits_{\substack{1 \leq i \leq n,\\ 1 \leq k \leq t_i}} i(1 - \omega^{-c_{i, k}} g^{e_i}) \cdot \prod\limits_{\substack{1 \leq i \leq n,\\ 1 \leq k \leq t'_i}} i(1 - \omega^{-d_{i, k}} g^{a_i}) = 0, 
\end{equation}
where the product is with respect to the convolution operator $*$ in $C(\mathbb{F}_p^n)$.

Therefore, equivalently, by using Fourier transform, we have to prove that
\begin{equation}\label{eq-Fourier}
\prod\limits_{\substack{1 \leq i \leq n,\\ 1 \leq k \leq t_i}} F(i(1 - \omega^{-c_{i, k}} g^{e_i})) \cdot \prod\limits_{\substack{1 \leq i \leq n,\\ 1 \leq k \leq t'_i}} F(i(1 - \omega^{-d_{i, k}} g^{a_i})) = 0.
\end{equation}

Notice however that if $\langle v, w \rangle = a$, then $F(i(1 - \omega^{-a} g^{v}))(w) = 0$, and we know that for every vector $v \in \mathbb{F}_p^m$ there exists $i\in [n],\ k\in [t_i]$ such that $\langle v, e_i \rangle = c_{i, k}$ or $i\in[n],\ k\in[t'_i]$ such that $\langle v, a_i \rangle = d_{i, k}$, so indeed we get that \eqref{eq-Fourier} holds,
which proves the implication $\text{(P1)} \to \text{(P3)}$.

\medskip

Let us continue with the proof of the direction $\text{(P3)} \to \text{(P1)}$.

\noindent
Let us assume that \eqref{eq-iomega} holds, then, equivalently, \eqref{eq-Fourier} also holds.

%$$\prod\limits_{1 \leq i \leq n, 1 \leq k \leq t_i} i(1 - \omega^{-c_{i, k}} g^{e_i}) \cdot \prod\limits_{1 \leq i \leq n, 1 \leq k \leq t'_i} i(1 - \omega^{-d_{i, k}} g^{a_i}) = 0 $$.

%This means by Fourier transform that $\prod\limits_{1 \leq i \leq n, 1 \leq k \leq t_i} F(i(1 - \omega^{-c_{i, k}} g^{e_i})) \cdot \prod\limits_{1 \leq i \leq n, 1 \leq k \leq t'_i} F(i(1 - \omega^{-d_{i, k}} g^{a_i})) = 0.$.

\noindent
On the other hand if $\langle v, w \rangle \neq a$, then $$F(i(1 - \omega^{-a} g^{v}))(w) = 1 - \omega^{-a} \cdot \omega^{\langle v, w \rangle} \neq  0,$$
so this means that for every vector $v \in \mathbb{F}_p^m$ either there exist $i\in[n],\ 1 k\in[t_i] $ such that $\langle v, e_i \rangle = c_{i, k}$ or there exist $i\in[n],\ k\in [t'_i] $ such that $\langle v, a_i\rangle = d_{i, k} $ which proves the direction $\text{(P3)} \to \text{(P1)}$.

\medskip

Next, let us prove that statement (P4) follows from the statements (P1)-(P2)-(P3).

\noindent
Let us write  
$$\prod_{1 \leq i \leq n} (1 - g^{e_i})^{t_i} \cdot   \prod_{1 \leq i \leq n} (1 - g^{a_i})^{t'_i} = 
\sum_{v \in G} r_v g^v \in \mathbb{Z}[\mathbb{F}_p^n],$$
then our aim is to show that $p \mid  r_v$ for every $v \in \mathbb{F}_p^n$.

Let us write
\begin{equation*}
    T := \prod\limits_{\substack{1 \leq i \leq n,\\ 1 \leq k \leq t_i}} (1 - \omega^{-c_{i, k}} g^{e_i}) \cdot \prod\limits_{\substack{1 \leq i \leq n,\\ 1 \leq k \leq t'_i}} (1 - \omega^{-d_{i, k}} g^{a_i})  = \sum_{v \in \mathbb{F}_p^n} \left( \sum_{0 \leq i \leq p-1} r_{v, i} \cdot \omega^i \right) g^v,
\end{equation*}
where we certainly have $r_v = \sum\limits_{0 \leq i \leq p-1} r_{v, i}$.

Since we know that $T = 0$, we have $\sum\limits_{0 \leq i \leq p-1} r_{v, i} \cdot \omega^i = 0$ for each $v \in G$.
There is only one linear dependence among the $p$-th roots of unity over $\mathbb{Z}$, namely $\sum_{0 \leq i \leq p-1} \omega^{i} = 0$. This indeed yields that $p \mid \sum\limits_{0 \leq i \leq p-1} r_{v, i} = r_v$, which finishes the proof of this implication.

\medskip

Finally, let us prove  that statement (P5) follows from statement (P4).

\noindent
Note that we can express property (P4) in the following way: 
$$\prod_{1 \leq i \leq n} w_i^{t_i} \cdot \prod_{1 \leq i \leq n} (1 - g^{a_i})^{t_i'} = 0 \in \mathbb{F}_p[\mathbb{F}_p^n].$$

For $1 \leq i \leq n $ let us consider the multiset $S_i$ of formal variables $w_j$ which contains $0 \leq a_{i, j} \leq p-1$ pieces of $w_j$.

If we write $g^{e_i}$ as $1- w_i$ for all $1 \leq i \leq n$, then we can express $1 - g^{a_i}$ as a polynomial of the variables $w_i$ such that 
$$1 - g^{a_i} =  \sum_{S':\ \emptyset\ne S' \subseteq S_i} (-1)^{|S'| + 1} \cdot \prod(S').$$

Notice that the the smallest degree $1$ part in the variables $w_i$ of $1 - g^{a_i} $ is $\sum\limits_{1 \leq j \leq n} a_{i, j} w_j$.

Since 
$$\prod_{1 \leq i \leq n} w_i^{t_i} \cdot \prod_{1 \leq i \leq n} (1 - g^{a_i})^{t'_i} = 0 \in \mathbb{F}_p[\mathbb{F}_p^n], $$
we indeed get that every monomial in 
$$\prod_{1 \leq i \leq n} \left( \sum_{1 \leq j \leq n} a_{i, j} w_j \right)^{t'_i} \cdot \prod_{1 \leq i \leq n} w_i^{t_i}$$
must have an exponent which is at least $p$, which indeed shows that $$\deg(\overline{f}) < \sum_{1 \leq i \leq n} t_i + \sum_{1 \leq i \leq n} t'_i.$$

\end{proof}

Let us formulate as a corollary the special case when we set $t_i = t'_i = 1$ (for $i\in [n]$) and $c_{i, 1} = d_{i, 1} = 0$ (for $i\in [n]$), since this is the case corresponding directly to the Alon-Jaeger-Tarsi conjecture. 

%We have the following corollary of proposition\ref{properties} in the simplest case $t_i = t'_i = 1, 1 \leq i \leq n$ and $c_{i, 1} = d_{i, 1} = 1, 1 \leq i \leq n $:

\begin{corollary}\label{properties2}
Let $p> 3$ be a prime, $n \geq 1$ be an integer and let $M$ be a nonsingular $n\times n$ matrix over $\mathbb{F}_p$. Let us denote the row vectors of $M$  by $a_1, \dots, a_n$ and let the standard basis vectors be denoted by $e_1, \dots, e_n$. Then the following three statements are equivalent:
\begin{itemize}
\item[(P'1)]  There does not exist a vector $x \in \mathbb{F}_p^n$ such that $x_i \neq 0$ for each $1 \leq i \leq n$  and $(Mx)_i \neq 0$ for each $1 \leq i \leq n$.

\item[(P'2)] The polynomial $$h(x_1, \dots, x_n) = \prod\limits_{1 \leq i \leq n} \left( \sum\limits_{1 \leq j \leq n} a_{i, j} x_j \right) \cdot \prod\limits_{1 \leq i \leq n} x_i$$
vanishes for all entries $x_1 = z_1, \dots, x_n = z_n$ where $z_1, \dots, z_n \in \mathbb{F}_p$.

\item[(P'3)] The element 
$$ \prod\limits_{1 \leq i \leq n} (1 - g^{e_i}) \cdot \prod\limits_{1 \leq i \leq n} (1 - g^{a_i}) \in \mathbb{Z}[\mathbb{F}_p^n]$$ is zero.

\end{itemize}

\noindent
The following statement follows from the equivalent statements \emph{(P'1)-(P'2)-(P'3)}:

\begin{itemize}

\item[(P'4)] The element $\prod\limits_{1 \leq i \leq n} (1 - g^{e_i}) \cdot \prod\limits_{1 \leq i \leq n} (1 - g^{a_i}) \in \mathbb{F}_p[\mathbb{F}_p^n]$ is zero.

\end{itemize}

\noindent
The following statement follows from the statement \emph{(P'4)}:

\begin{itemize}
\item[(P'5)] For the polynomial $f$ defined as $$f(x_1, \dots, x_n) = \prod_{1 \leq i \leq n} \left( \sum_{1 \leq j \leq n} a_{i, j} x_j \right) \cdot \prod_{1 \leq i \leq n} x_i$$ we have $\deg(\overline{f}) < 2n$.
\end{itemize}
That is, we have $(\emph{P'}1)\leftrightarrow (\emph{P'}2) \leftrightarrow (\emph{P'}3) \to (\emph{P'}4) \to (\emph{P'}5)$.
\end{corollary}

\bigskip

We propose the following conjecture, already mentioned in the introduction, which claims that if $p > 3$, then in fact property (P'5) is not weaker than
property (P'4).

\begin{conjecture2}
Let $p > 3$ be a prime and let $M$ be a nonsingular $n \times n$ matrix over $\mathbb{F}_p$ with
row vectors $a_1, \dots,a_n$, finally, let the standard basis vectors of $\mathbb{F}_p^n$ be $e_1,\dots,e_n$.

\noindent
Then the $\mathbb{Z}$ group ring identity 
$$\prod_{1 \leq i \leq n} (1 - g^{e_i}) \cdot   \prod_{1 \leq i \leq n} (1 - g^{a_i})  = 0 \in \mathbb{Z}[G]$$
follows from the mod $p$ group ring identity 
$$\prod_{1 \leq i \leq n} (1 - g^{e_i}) \cdot   \prod_{1 \leq i \leq n} (1 - g^{a_i})  = 0 \in \mathbb{F}_p[G].$$
%With the usual setup the $\mathbb{Z}$ group ring identity 
%$$\prod_{1 \leq j \leq n} (1 - g^{e_j}) \cdot   \prod_{1 \leq j \leq n} (1 - g^{a_j})  = 0 \in \mathbb{Z}[G]$$
%follows from the mod $p$ group ring identity 
%$$\prod_{1 \leq j \leq n} (1 - g^{e_j}) \cdot   \prod_{1 \leq j \leq n} (1 - g^{a_j})  = 0 \in \mathbb{F}_p[G].$$
\end{conjecture2}

\subsection{Equivalent characterizations of properties (P'1)-(P'4)}

In this subsection we consider two properties of two bases of $\mathbb{F}_p^n$:
\begin{itemize}
\item The stronger one is that it is a counterexample for the Alon-Jaeger-Tarsi conjecture, which is equivalent with the properties (P'1)-(P'2)-(P'3) of Corollary~\ref{properties2}.

\item The weaker one is property (P'4) which is the same identity as property (P'3) just in the group ring with mod $p$ coefficients.

\end{itemize}

\bigskip

\bigskip

Now, we express the equivalent statements (P'1)-(P'2)-(P'3) in yet another way:

Let us denote the inverse of the matrix $M$ by $M'$ and let its columns be $a'_1,  \dots, a'_n$.

\noindent
Notice that $\langle a'_1, x\rangle , \dots, \langle a'_n, x\rangle$ parametrize a coordinate system on $\mathbb{F}_p^n$ such that the vectors $a_1, \dots, a_n$ lie on the coordinate lines.

Now, we prove a proposition giving an equivalent characterisation of properties (P'1)-(P'2)-(P'3).

\begin{proposition}\label{equivalent123}
For a nonsingular matrix $M \in M^{n \times n}(\mathbb{F}_p)$ properties (P'1)-(P'2)-(P'3) are equivalent to the following statement:

\smallskip

If $f \in C( \mathbb{F}_p^n)$ such that for every $x \in \mathbb{F}_p$ and $i \in [n]$ we have 
$$\sum\limits_{\substack{v \in \mathbb{F}_p^n,\\ v_i = x}} f(v) = 0$$
and $g \in C( \mathbb{F}_p^n)$ such that for every $x \in \mathbb{F}_p$ and $i \in [n]$ we have 
$$\sum\limits_{\substack{v \in \mathbb{F}_p^n,\\ \langle a'_i, v \rangle = x}} g(v) = 0,$$
then $\sum\limits_{ v \in \mathbb{F}_p^n} f(v) \cdot g(v) = 0$.
\end{proposition}

\begin{proof}

It can be easily seen that for $f \in C( \mathbb{F}_p^n)$, for every $x \in \mathbb{F}_p$ and $i\in [n]$ the equation $\sum\limits_{v \in \mathbb{F}_p^n, v_i = x} f(v) = 0$ holds if and only if
$F(f)$ vanishes on all the coordinate subspaces $v_i = 0$ for some $i\in[ n]$.

Similarly, for $g \in C( \mathbb{F}_p^n)$, for every $x \in \mathbb{F}_p$ and $i\in [n]$  the equation $\sum\limits_{v \in \mathbb{F}_p^n, \langle a'_i, v \rangle = x} g(v) = 0$ holds if and only if
$F(g)$ vanishes on all the coordinate subspaces $\langle a_i, v \rangle = 0$ for some $i\in [n]$.

Now, if (P'1)-(P'2)-(P'3)  hold, then for every vector $v \in \mathbb{F}_p^n$ we have an index $i\in [n]$ such that $v_i = 0$ or $\langle a_i, v \rangle = 0$, which means that 
if $f$ and $g$ satisfy our conditions, then 
$$\sum\limits_{v  \in \mathbb{F}_p^n} f(v) \cdot g(v) = \sum_{v  \in \mathbb{F}_p^n} F(f)(v) \cdot F(g)(v) = 0,$$ 
which proves one direction.

For the other direction, let us assume that there exists a vector $v \in \mathbb{F}_p^n$ such that for every $i\in [n]$ we have $v_i \neq 0$ and $\langle a_i, v \rangle \neq 0$. Then for the functions
$f (w): = \omega^{\langle w, v \rangle}$ and $g(w) := \omega^{-\langle w, v \rangle}$ we have 
$$\sum_{v  \in \mathbb{F}_p^n} f(v) \cdot g(v)  = p^n \neq 0$$
and on the other hand  for every $x \in \mathbb{F}_p$ and $i\in [n]$ we have 
$$\sum\limits_{\substack{v \in \mathbb{F}_p^n,\\ v_i = x}} f(v) = 0$$ and for every $x \in \mathbb{F}_p$ and $1 \leq i \leq n$ we have  $$\sum\limits_{\substack{v \in \mathbb{F}_p^n,\\ \langle a'_i, v\rangle = x}} g(v) = 0,$$
which proves the other direction.

\end{proof}

Let us consider the space of functions $\mathbb{F}_p^n \to \mathbb{F}_p$ and let us denote it by $C_p(\mathbb{F}_p^n)$,
the space $C_p(\mathbb{F}_p^n)$ is a  vector space over $\mathbb{F}_p$ and its zero element is the constant zero function that will be denoted simply by $0$.

In the following we prove a similar equivalent characterisation for property (P'4):

\begin{proposition}\label{equivalent4}
For a nonsingular matrix $M \in M^{n \times n}(\mathbb{F}_p)$ property (P'4) is equivalent to the following statement:
\bigskip
If $f \in C_p( \mathbb{F}_p^n)$ such that for every $x \in \mathbb{F}_p$ and $i \in [n]$ we have 
$$\sum\limits_{\substack{v \in \mathbb{F}_p^n,\\ v_i = x}} f(v) = 0$$
and $g \in C_p( \mathbb{F}_p^n)$ such that for every $x \in \mathbb{F}_p$ and $i \in [n]$ we have 
$$\sum\limits_{\substack{v \in \mathbb{F}_p^n,\\ \langle a'_i, v\rangle = x}} g(v) = 0,$$ then
$$\sum_{ v \in \mathbb{F}_p^n} f(v) \cdot g(v) = 0.$$
\end{proposition}

\begin{proof}

Note that every function $f \in C_p(\mathbb{F}_p^n)$ can be written uniquely in the form 
$$f = \sum_{0 \leq b_1 \leq p-1, \dots, 0 \leq b_n \leq p-1} c_{b_1, \dots, b_n} \cdot \prod_{1 \leq i \leq n} x_i^{b_i},$$
where $ c_{b_1, \dots, b_n} \in \mathbb{F}_p$.

Let us define the operators $\delta_i : C_p(\mathbb{F}_p^n) \to C_p(\mathbb{F}_p^n)$  by 
$$ \delta_i(f)(v) := f(v) - f(v + e_i)$$
and operators $\delta'_i : C_p(\mathbb{F}_p^n) \to C_p(\mathbb{F}_p^n)$ by 
$$ \delta'_i(f)(v) = f(v) - f(v + a_i).$$
It is easy to see that the operators $\delta_i, \delta'_i\ (i\in [n])$ pairwise commute.

For more about the connections of operators like these and the so-called {\it functional degree} of functions between groups, see \cite{Aichinger}.

Observe that property (P'4), so the condition that
$$\prod_{1 \leq i \leq n} (1 - g^{e_i}) \cdot \prod_{1 \leq i \leq n} (1 - g^{a_i}) \in \mathbb{F}_p[\mathbb{F}_p^n]$$
is zero is equivalent to the statement that
$$ \delta_1 \circ \cdots \circ \delta_n \circ \delta'_1 \circ \cdots \circ \delta'_n(f) = 0,$$ 
for every $f \in C_p(\mathbb{F}_p^n)$.

Indeed, if 
$$ \prod_{1 \leq i \leq n} (1 - g^{e_i}) \cdot \prod_{1 \leq i \leq n} (1 - g^{a_i})  = \sum_{w \in \mathbb{F}_p^n} t_w g^w \in \mathbb{Z}[\mathbb{F}_p^n],$$
then
$$\delta_1 \circ \cdots \circ \delta_n \circ \delta'_1 \circ \cdots \circ \delta'_n(f)(v)  = \sum_{w \in \mathbb{F}_p^n} t_w f(v + w) = 0 \in C_p(\mathbb{F}_p^n).$$ 

This means that each $t_w$ is divisible by $p$ if and only if $\delta_1 \circ \cdots \circ \delta_n \circ \delta'_1 \circ \cdots \circ \delta'_n(f) = 0$, for every $f \in C_p(\mathbb{F}_p^n)$, which proves our claim.

\bigskip

We claim in the following that the image  
$$\imm(\delta_1 \circ \dots \circ \delta_n) \subseteq C_p(\mathbb{F}_p^n) $$
consists of the $(p-1)^n$ functions of the form $$\sum_{0 \leq b_1 \leq p-2, \dots, 0 \leq b_n \leq p-2} c_{b_1, \dots, b_n} \cdot \prod_{1 \leq i \leq n} x_i^{b_i},$$
where $ c_{b_1, \dots, b_n} \in \mathbb{F}_p$.

This follows by observing that if $ 1 \leq b'_1 \leq p-1, \dots, 1 \leq b'_n \leq p-1$, then:
$$\delta_1 \circ \cdots \circ \delta_n(\prod_{1 \leq i \leq n} x_i^{b'_i}) = \prod_{1 \leq i \leq n} b'_i \cdot \prod_{1 \leq i \leq n} x_i^{b'_i-1} + g(x_1, \dots, x_n),$$
where $g(x_1, \dots, x_n)$ is a reduced polynomial and $\deg g < \sum\limits_{1 \leq i \leq n} (b'_i - 1)$.

Notice that the kernel  
$$\kerr(\delta_1 \circ \dots \circ \delta_n)\subseteq C_p(\mathbb{F}_p^n) $$ consists of the functions  of the form 
$$\sum_{0 \leq b_1 \leq p-1, \dots, 0 \leq b_n \leq p-1} c_{b_1, \dots, b_n} \cdot \prod_{1 \leq i \leq n} x_i^{b_i},$$
where if  $ c_{b_1, \dots, b_n} \neq 0$, then one of the numbers $b_1, \dots, b_n$ is equal to $0$.

Indeed, these functions are definitely in the kernel and they form a subspace of dimension $p^n - (p-1)^n$ which is the same as $\dim(C_p(\mathbb{F}_p^n) ) - \dim(\imm(\delta_1 \circ \cdots \circ \delta_n))$, thus  the kernel consists of only these functions.

Notice that for every function $f \in \kerr(\delta_1 \circ \dots \circ \delta_n)$ and every monomial $\prod\limits_{1 \leq i \leq n} x_i^{b_i}$ such that $0 \leq b_i \leq p-2$ we have
$$\Coeff\left( \prod_{1 \leq i \leq n} x_i^{p-1},  \overline{\prod_{1 \leq i \leq n} x_i^{b_i} \cdot f}\right) = 0.$$

In other words, if $f \in  \kerr(\delta_1 \circ \dots \circ \delta_n))$ and $g \in \imm(\delta_1 \circ \dots \circ \delta_n)$, then
$$\Coeff\left( \prod_{1 \leq i \leq n} x_i^{p-1}, f \cdot g\right) = 0,$$
so $\sum\limits_{v \in \mathbb{F}_p^n} f(v) \cdot g(v) = 0$.

Since the matrix $M$ is invertible we get that the image of the map $\delta'_1 \circ \dots \circ \delta'_n$ consists of the functions of the form $$\sum_{0 \leq b_1 \leq p-2, \dots, 0 \leq b_n \leq p-2} c_{b_1, \dots, b_n} \cdot \prod_{1 \leq i \leq n} \langle a'_i, x\rangle^{b_i},$$
where $ c_{b_1, \dots, b_n} \in \mathbb{F}_p$.

As $\imm(\delta'_1 \circ \dots \circ \delta'_n) \subseteq \kerr(\delta_1 \circ \dots \circ \delta_n)$ we get that if $f  \in  \imm(\delta'_1 \circ \dots \circ \delta'_n)$ and $g \in \imm(\delta_1 \circ \dots \circ \delta_n)$, then $$\sum_{v \in \mathbb{F}_p^n} f(v) \cdot g(v) = 0.$$

The following lemma gives another description of the functions in $\imm(\delta_1 \circ \dots \circ \delta_n) \subseteq C_p(\mathbb{F}_p^n)$ which will prove our proposition completely:

\begin{lemma}
A function $f \in C_p(\mathbb{F}_p^n)$ is in $\imm(\delta_1 \circ \dots \circ \delta_n)$ if and only if for each index $i \in [n]$ and for every vector $v \in \mathbb{F}_p^n$ one has $\sum\limits_{0 \leq t \leq p-1} f(v + t \cdot e_i) = 0$.
\end{lemma}
\begin{proof}

If a function $f \in C_p(\mathbb{F}_p^n)$ is in $\imm(\delta_1 \circ \cdots \circ \delta_n)$, then for each index $1 \leq i \leq n$ one has $f \in \imm(\delta_i)$, so there is a function $F \in C_p(\mathbb{F}_p^n)$ such that $f(v) = F(v + e_i) - F(v)$ for each $v \in \mathbb{F}_p^n$ and  we get the telescoping sum
$$\sum_{0 \leq t \leq p-1} f(v + t \cdot e_i) = \sum_{0 \leq t \leq p-1} \left[F(v + (t+1) \cdot e_i) - F(v + t \cdot e_i) \right] = 0.$$

\medskip

For the other direction let us assume that $f \in C_p(\mathbb{F}_p^n)$ and for each index $1 \leq i \leq n$ and for every vector $v \in \mathbb{F}_p^n$ one has $\sum\limits_{0 \leq t \leq p-1} f(v + t \cdot e_i) = 0$.

Let us write $f \in C_p(\mathbb{F}_p^n)$ in the reduced form 
$$\sum_{0 \leq b_1 \leq p-1, \dots, 0 \leq b_n \leq p-1} c_{b_1, \dots, b_n} \cdot \prod_{1 \leq i \leq n} x_i^{b_i},$$ where $ c_{b_1, \cdots, b_n} \in \mathbb{F}_p$.

Let us pick an index $1 \leq i \leq n$ and  write 
$$f = \sum_{0 \leq j \leq p-1} x_i^j \cdot g_j(x_1, \dots, x_{i-1}, x_{i+1}, \dots, x_n),$$
where the $g_j$ are reduced polynomials in the variables $x_1, \dots, x_{i-1}, x_{i+1}, \dots, x_n$.

If $g_{p-1}(x_1, \dots, x_{i-1}, x_{i+1}, \dots, x_n) \neq 0$, then there are values $z_1, \dots, z_{i-1}, z_{i+1}, \dots, z_n \in \mathbb{F}_p$ such that $g_{p-1}(z_1, \dots, z_{i-1}, z_{i+1}, \dots, z_n) \neq 0$.

Notice however that 
$$\sum_{0 \leq j \leq p-1} f(z_1 , \dots, z_{i-1}, j, z_{i+1}, \dots, z_n) = g_{p-1}(z_1, \dots, z_{i-1}, z_{i+1}, \dots, z_n) \neq 0,$$
which is impossible by the assumption.

Hence, for every index $1 \leq i \leq n$ we have $g_{p-1}(x_1, \dots, x_{i-1}, x_{i+1}, \dots, x_n) = 0$ which means that $f$ is of the form 
$$\sum_{0 \leq b_1 \leq p-2, \dots, 0 \leq b_n \leq p-2} c_{b_1, \dots, b_n} \cdot \prod_{1 \leq i \leq n} x_i^{b_i},$$
where $ c_{b_1, \dots, b_n} \in \mathbb{F}_p$.

However, by our earlier observations this indeed means that $f \in \imm(\delta_1 \circ \dots \circ \delta_n)$.
\end{proof}

Since the matrix $M$ is nonsingular, we also get that a function $g \in C_p(\mathbb{F}_p^n)$ is in $\imm(\delta'_1 \circ \dots \circ \delta'_n)$ if and only if for each index $1 \leq i \leq n$ and for every vector $v \in \mathbb{F}_p^n$ one has $\sum\limits_{0 \leq t \leq p-1} f(v + t \cdot a_i) = 0$.

This completes the proof of the proposition.
\end{proof}

\begin{remark}
In the interpretation of property (P'4) in Proposition~\ref{equivalent4} the function space $C_p( \mathbb{F}_p^n)$ can be replaced by the space of $A$-valued functions on $\mathbb{F}_p^n$, $C_A(\mathbb{F}_p^n)$, where $A$ is an arbitrary $\mathbb{F}_p$ algebra.
\end{remark}

\subsection{Another equivalent characterisation of mod $p$ group ring properties}

In this subsection we investigate the mod $p$ group ring property (P4), that is,
$$\prod\limits_{1 \leq i \leq n} (1 - g^{e_i})^{t_i} \cdot \prod\limits_{1 \leq i \leq n} (1 - g^{a_i})^{t'_i} = 0 \in \mathbb{F}_p[\mathbb{F}_p^n],$$
where
$0 \leq t_i \leq p-1,\ 0 \leq t'_i \leq p-1\ (i\in[ n])$.

We have proved equivalent characterizations of this property in the special case $t_i = t'_i = 1\ (i\in [n])$, but 
we will need the following simple lemma in the general case which will be very useful in the proof of our main theorem:

\begin{lemma}\label{testmodp}
For $i \in [n]$ let  $0 \leq t_i \leq p-1,\ 0 \leq t'_i \leq p-1$ be integers and let $\lambda_1, \dots, \lambda_n, \mu_1, \dots, \mu_n \in \mathbb{F}_p^*$ be residues.
Let us consider the following property: For all integers $0 \leq k_i \leq t_i,\ 0 \leq k'_i \leq t'_i$, if we have the function $f: \mathbb{F}_p^n \to \mathbb{F}_p$ given by the formula $$f\left(\sum\limits_{1 \leq i \leq n} (\ell_i - k_i)  \lambda_i e_i\right) =   
(-1)^{\sum_{1 \leq i \leq n} \ell_i} \cdot \prod_{1 \leq i \leq n} { \binom{t_i}  {\ell_i}},$$
where $0 \leq \ell_i \leq t_i\ (i \in [n])$
and $f$ takes values $0$ otherwise and similarly $g: \mathbb{F}_p^n \to \mathbb{F}_p$ is given by the formula 
$$g\left(\sum\limits_{1 \leq i \leq n} (k'_i - \ell'_i)  \mu_i a_i\right) =   
(-1)^{\sum_{1 \leq i \leq n} \ell'_i} \cdot \prod_{1 \leq i \leq n} { \binom{t'_i}{ \ell'_i}},$$
where $0 \leq \ell'_i \leq t'_i\ (i \in [n])$, then we have $\langle f, g \rangle = 0$.

This property is equivalent to  property (P4), that is, to 
$$ \prod\limits_{1 \leq i \leq n} (1 - g^{e_i})^{t_i} \cdot \prod\limits_{1 \leq i \leq n} (1 - g^{a_i})^{t'_i} = 0 \in \mathbb{F}_p[\mathbb{F}_p^n].$$
\end{lemma}

\begin{proof}
For the integers   $0 \leq k_i \leq t_i, 0 \leq k'_i \leq t'_i$ the property of the lemma is equivalent to the fact that the coefficient
of 
$$g^{\sum\limits_{1 \leq i \leq n} k_i \lambda_i e_i + \sum\limits_{1 \leq i \leq n} k'_i \mu_i a_i}$$
is $0$ in 
$$\prod_{1 \leq i \leq n} (1 - g^{\lambda_i e_i})^{t_i} \cdot   \prod_{1 \leq i \leq n} (1 -   g^{\mu_i a_i})^{t'_i} \in \mathbb{F}_p[\mathbb{F}_p^n],$$
so is equivalent to 
$$\prod_{1 \leq i \leq n} (1 - g^{\lambda_i e_i})^{t_i} \cdot   \prod_{1 \leq i \leq n} (1 -   g^{\mu_i a_i})^{t'_i}  = 0  \in \mathbb{F}_p[\mathbb{F}_p^n].$$

Notice on the other hand that 
$$\prod_{1 \leq i \leq n} (1 - g^{\lambda_i e_i})^{t_i} \cdot   \prod_{1 \leq i \leq n} (1 -   g^{\mu_i a_i})^{t'_i}\quad \text{ and } \quad
 \prod_{1 \leq i \leq n} (1 - g^{ e_i})^{t_i} \cdot   \prod_{1 \leq i \leq n} (1 -   g^{a_i})^{t'_i}$$ differs in an invertible element in $\mathbb{F}_p[\mathbb{F}_p^n]$ since 
$$(1 -  g^{\lambda_i e_i}) = (1 -  g^{e_i}) \cdot \sum_{0 \leq j \leq \lambda_i -1} g^{j e_i}\quad
\text{ and } 
\quad (1 - g^{\lambda_i e_i})  \cdot \sum_{0 \leq j \leq \overline{\lambda_i^{-1}} -1}  g^{j \lambda_i e_i} = (1 -  g^{e_i}),$$
where $\overline{\lambda_i^{-1}}$ the unique integer in the interval $[1, p]$ which is congruent to $\lambda_i^{-1}$ mod $p$.

\medskip

Hence,
$$ \prod_{1 \leq i \leq n} (1 - g^{ e_i})^{t_i} \cdot   \prod_{1 \leq i \leq n} (1 -   g^{a_i})^{t'_i}  = 0 \in \mathbb{F}_p[\mathbb{F}_p^n]$$
is equivalent to 
$$\prod_{1 \leq i \leq n} (1 - g^{\lambda_i e_i})^{t_i} \cdot   \prod_{1 \leq i \leq n} (1 -   g^{\mu_i a_i})^{t'_i}  = 0  \in \mathbb{F}_p[\mathbb{F}_p^n].$$
\end{proof}

\subsection{Duality of polynomial coefficients}

In this section our aim is to compare the polynomial descriptions of properties (P'1) and (P'4).

We have seen in the proof of Proposition~\ref{equivalent4} that (P'4), that is, the mod $p$ group ring identity is equivalent to the vanishings of the coefficients 
$$\Coeff\left(\prod_{1 \leq j \leq n} x_j^{t_j}, \overline{\prod_{1 \leq j \leq n} (\sum_{1 \leq i \leq n} a_{i, j} x_i)^{b_j}}\right) = 0$$
for every $0 \leq b_j \leq p-2, \ 1 \leq t_j \leq p-1\ (j\in [n]).$

In fact this statement was written in the coordinate systems $x_i\ (i\in [n])$ and $\langle a'_i, x\rangle\ (i\in [n])$, however if we change the roles of the two coordinate systems we get this form, which will be more convenient in the following.

Equivalently, we may express these vanishing statements in the form that if $0 \leq b_j \leq p-2$ and $1 \leq t_j \leq p-1$, then 
$$\left \langle \prod_{1 \leq j \leq n} \left(\sum_{1 \leq i \leq n} a_{i, j} x_i\right)^{b_j},  \prod_{1 \leq j \leq n} x_j^{p-1 - t_j}\right\rangle = 0,$$
with the standard scalar product $\langle \cdot,\cdot \rangle$ of functions.

Among these conditions let us call {\it top degree conditions} the ones when $\sum\limits_{1 \leq j \leq n}t_j = \sum\limits_{1 \leq i \leq n}b_j$.

Let us prove first the following lemma which shows that we can also express property (P'1) by the vanishing of a bunch of similar polynomial coefficients:

\begin{lemma}\label{Alontpol}
Property \emph{(P'1)} is equivalent to the vanishings
$$\left \langle \prod_{1 \leq i \leq n} \left(\sum_{1 \leq j \leq n} a_{i, j} x_j\right)^{r_i},  \prod_{1 \leq i \leq n} x_i^{s_i}\right  \rangle = 0,$$
where $1 \leq r_i \leq p-1$ and $1 \leq s_i \leq p-1$.

In particular, under property \emph{(P'1)}, if $1 \leq r_i \leq p-1$, $0 \leq s'_i \leq p-2$ and $\sum\limits_{1 \leq i \leq n}r_i = \sum\limits_{1 \leq i \leq n}s'_i$, then
$$\Coeff\left(\prod_{1 \leq i \leq n} x_i^{s'_i}, \overline{\prod_{1 \leq i \leq n} (\sum_{1 \leq j \leq n} a_{i, j} x_j)^{r_i}}\right) = 0.$$
\end{lemma}
\begin{proof}
If property (P'1) holds, then indeed for every vector $x \in \mathbb{F}_p^n$ one of the coordinates $x_i\ (i\in [n])$ or $\sum\limits_{1 \leq j \leq n} a_{i, j} x_j\  (i\in [n])$ vanishes, so 
$$\left\langle \prod_{1 \leq i \leq n} \left(\sum_{1 \leq j \leq n} a_{i, j} x_j\right)^{r_i},  \prod_{1 \leq i \leq n} x_i^{s_i}\right\rangle  = 0$$
holds trivially.

For the other direction let us assume  that 
$$\left \langle \prod_{1 \leq i \leq n} \left(\sum_{1 \leq j \leq n} a_{i, j} x_j\right )^{r_i},  \prod_{1 \leq i \leq n} x_i^{s_i} \right \rangle = 0,$$ 
where $1 \leq r_i \leq p-1$ and $1 \leq s_i \leq p-1$.

For the sake of contradiction let us assume  that property (P'1) does not hold, then there is a vector $v \in \mathbb{F}_p^n$ such that
$v_i \neq 0\ (i\in [n])$ and $\sum\limits_{1 \leq j \leq n} a_{i, j} v_j \neq 0\ (i\in [n]$).

Let us consider the functions $f(x_1, \dots, x_n): = \prod\limits_{1 \leq i \leq n} \prod\limits_{r \neq v_i}(x_i - r)$
and $g(x_1, \dots, x_n) := \prod\limits_{1 \leq i \leq n} (\sum\limits_{1 \leq j \leq n} a_{i, j} x_j)$.

One can see that 
\begin{itemize}
    \item $f(x_1, \dots, x_n)$ is the linear combination of monomials of the form $\prod\limits_{1 \leq i \leq n} x_i^{s_i}$ where $1 \leq s_i \leq p-1$ and
    \item  $g(x_1, \dots, x_n)$ is a function
of the form $\prod\limits_{1 \leq i \leq n} (\sum\limits_{1 \leq j \leq n} a_{i, j} x_j)^{r_i}$ (where in fact  $r_i = 1$ for every $i$).
\end{itemize}

On the other hand $f(x) = 1$ if $x = v$ and $f(x) = 0$ otherwise and $g(v) \neq 0$, so we get $\langle f, g\rangle  \neq 0$ which is a contradiction.
\bigskip

Assume next that $1 \leq r_i \leq p-1$,  $0 \leq s'_i \leq p-2$ and $\sum\limits_{1 \leq i \leq n}r_i = \sum\limits_{1 \leq j \leq n}s'_i$.

\noindent
In this case it is easy to see that we have 
$$\Coeff\left (\prod_{1 \leq i \leq n} x_i^{s'_i}, \overline{\prod_{1 \leq i \leq n} \left(\sum_{1 \leq j \leq n} a_{i, j} x_j \right)^{r_i}}\right) = \left\langle\prod_{1 \leq i \leq n} \left(\sum_{1 \leq j \leq n} a_{i, j} x_j \right)^{r_i},  \prod_{1 \leq i \leq n} x_i^{p-1- s'_i}\right\rangle,$$ 
which proves the lemma.
\end{proof}

Again, among the conditions of the previous lemma let us call {\it top degree conditions} the ones that 
$$\Coeff\left (\prod_{1 \leq i \leq n} x_i^{s'_i}, \overline{\prod_{1 \leq i \leq n} \left(\sum_{1 \leq j \leq n} a_{i, j} x_j \right)^{r_i}}\right) = 0$$
if $1 \leq r_i \leq p-1$,  $0 \leq s'_i \leq p-2$ and $\sum\limits_{1 \leq i \leq n}r_i = \sum\limits_{1 \leq i \leq n}s'_i$.

Next, we show these top degree conditions already follow from the top degree conditions associated to property (P'4).

To see this, we prove a duality lemma between polynomial coefficients:

\begin{lemma}\label{duality}
For $i \in [n]$ let $0\leq r_i\leq p-1$ and $0\leq s_i\leq p-1$ be integers such that $\sum\limits_{1 \leq i \leq n}s_i = \sum\limits_{1 \leq i \leq n}r_i$.

Then we have 
$$\Coeff\left(\prod_{1 \leq i \leq n} x_i^{s_i}, \overline{\prod_{1 \leq i \leq n} \left(\sum_{1 \leq j \leq n} a_{i, j} x_j \right)^{r_i}}\right) = 0$$
if and only if 
$$\Coeff\left(\prod_{1 \leq j \leq n} x_j^{r_j}, \overline{\prod_{1 \leq j \leq n} \left(\sum_{1 \leq i \leq n} a_{i, j} x_i \right)^{s_j}}\right) = 0.$$
\end{lemma}
\begin{proof}
Let us write $Q := \sum\limits_{1 \leq i \leq n}s_i = \sum\limits_{1 \leq i \leq n}r_i$ and  consider the polynomial
$$P(x_1, \dots, x_n, y_1, \dots, y_n) = \left(\sum_{\substack{1 \leq i \leq n,\\ 1 \leq j \leq n}} a_{i,j}x_j y_i\right)^Q.$$

Observe that
$$\Coeff\left(\prod_{1 \leq i \leq n} x_i^{s_i}, \overline{\prod_{1 \leq i \leq n} \left(\sum_{1 \leq j \leq n} a_{i, j} x_j \right)^{r_i}}\right)  = \Coeff\left(\prod_{1 \leq i \leq n} x_i^{s_i} \cdot
\prod_{1 \leq i \leq n} y_i^{r_i} , P\right) \cdot \frac{\prod_{1 \leq i \leq n} r_i!}{Q!}$$
and similarly,
$$\Coeff\left(\prod_{1 \leq j \leq n} x_j^{r_j}, \overline{\prod_{1 \leq j \leq n} \left(\sum_{1 \leq i \leq n} a_{i, j} x_i \right)^{s_j}}\right) = \Coeff\left(\prod_{1 \leq i \leq n} x_i^{s_i} \cdot
\prod_{1 \leq i \leq n} y_i^{r_i} , P\right) \cdot \frac{\prod_{1 \leq i \leq n} s_i!}{Q!}.$$

Hence,
\begin{multline*}
\prod_{1 \leq i \leq n} r_i!\cdot \Coeff\left(\prod_{1 \leq j \leq n} x_j^{r_j}, \overline{\prod_{1 \leq j \leq n} \left(\sum_{1 \leq i \leq n} a_{i, j} x_i \right)^{s_j}}\right)   =\\ \prod_{1 \leq i \leq n} s_i!\cdot \Coeff\left(\prod_{1 \leq i \leq n} x_i^{s_i}, \overline{\prod_{1 \leq i \leq n} \left(\sum_{1 \leq j \leq n} a_{i, j} x_j \right)^{r_i}}\right).
\end{multline*}

Since $0 \leq r_i,s_i \leq p-1$ for every $i\in [n]$  this proves the lemma.
\end{proof}

\begin{corollary}
From property \emph{(P'4)} it follows that if $1 \leq r_i \leq p-1$ and $1 \leq s_i \leq p-1$, then 
$$\left\langle \prod_{1 \leq i \leq n} \left(\sum_{1 \leq j \leq n} a_{i, j} x_j\right)^{r_i},  \prod_{1 \leq i \leq n} x_i^{s_i}\right\rangle = 0,$$
if $\sum\limits_{1 \leq i \leq n} (r_i + s_i) \leq (p-1)n$.

In particular, if $\sum\limits_{1 \leq i \leq n} (r_i + s_i) = (p-1)n$, then these conditions give back the top degree conditions of Lemma~\ref{Alontpol}.
\end{corollary}
\begin{proof}
Let us assume that property (P'4) holds.

If $\sum\limits_{1 \leq i \leq n} (r_i + s_i) < (p-1)n$, then the statement follows from the fact that 
$$\deg\left( \left(\sum\limits_{1 \leq j \leq n} a_{i, j} x_j \right)^{r_i} \cdot  \prod\limits_{1 \leq i \leq n} x_i^{s_i}\right) < (p-1)n.$$

If $\sum\limits_{1 \leq i \leq n} (r_i + s_i) = (p-1)n$, then 
$$\left\langle \prod_{1 \leq i \leq n} \left(\sum_{1 \leq j \leq n} a_{i, j} x_j\right)^{r_i},  \prod_{1 \leq i \leq n} x_i^{s_i} \right\rangle = 0$$
is equivalent to 
$$\Coeff\left(\prod_{1 \leq i \leq n} x_i^{p-1-s_i}, \prod_{1 \leq i \leq n} \left(\sum_{1 \leq j \leq n} a_{i, j} x_j\right)^{r_i} \right) = 0.$$

However, by Lemma~\ref{duality} we know that this is equivalent to 
$$\Coeff\left(\prod_{1 \leq j \leq n} x_j^{r_j}, \overline{\prod_{1 \leq j \leq n} \left(\sum_{1 \leq i \leq n} a_{i, j} x_i \right)^{p-1 - s_j}}\right) = 0,$$
which we indeed know from property (P'4).
\end{proof}

Although we could not complete this duality for all cases to conclude that the mod $p$ group ring identity is equivalent to the $\mathbb {Z}$ group ring identity, we will use a combinatorial interpretation of this duality phenomenon to prove the Alon-Jaeger-Tarsi conjecture for primes $61<p\ne 79$.

\section{Proof of Theorems A and B}\label{sec-proof}

\noindent
In this section we prove our main theorem:

\begin{theorem}\label{mainth}
(i) Let $p$ a prime and $k$ an integer such that $p > k \geq 1$ and $N_k(p) \cdot S_k(p) < p-1$ hold. Let $M \in M^{n \times n}(\mathbb{F}_p)$ be an arbitrary nonsingular matrix. Then the matrix $M$ does not satisfy property (P4), that is,
$$ \prod_{1 \leq i \leq n} (1 - g^{e_i})^k \cdot \prod_{1 \leq i \leq n} (1 - g^{a_i})^k \neq 0 \in \mathbb{F}_p[\mathbb{F}_p^n].$$

(ii) Let $61 < p \neq 79$ be a prime  and let $c_i, d_i \in \mathbb{F}_p$ be arbitrary residues. Then there exists a vector $x \in \mathbb{F}_p^n$ such that $x_i \neq c_i$ for each $1 \leq i \leq n$ and $(Mx)_i \neq d_i$ for each $1 \leq i \leq n$.

(iii) Assume that $p$ is a prime and $k \geq 2$ is an integer such that $p > k^{K k}$ (where $K$ is the universal constant from Lemma~\ref{Skepsilon}). For each $1\leq i\leq n$ let $c_{i,\ell}\in \mathbb{F}_p$ be $k$ different residues ($\ell \in [k]$) and let $d_{i,\ell}\in \mathbb{F}_p$ be $k$ different residues ($\ell\in [k]$). Then  there exists a vector $x \in \mathbb{F}_p^n$ such that $x_i \neq c_{i, \ell}$ for each $1 \leq i \leq n,\ 1 \leq \ell \leq k$ and $(Mx)_i \neq d_{i, \ell}$ for each $1 \leq i \leq n,\ 1 \leq \ell \leq k$.
\end{theorem}

\begin{remark}
Theorem 1, parts (ii) and (iii) are equivalent to Theorem A and Theorem B, respectively.
\end{remark}

\begin{proof}

For part (i), notice that by Corollary~\ref{goodsubsets} we can fix subsets $A_i,B_i \subseteq \mathbb{F}_p^*\ (i\in [n])$ such that for every $i\in [n]$ the set $A_i$ is $S_k$-type and the set $B_i$ is $N_k$-type, furthermore, if $y_i \in A_i\  (i\in [n])$, $z_i \in B_i\ (i\in [n])$, then $\sum\limits_{1 \leq i \leq n} y_i e_i \neq \sum\limits_{1 \leq i \leq n} z_i a_i$.

We will prove the statement by induction. Observe that in the case $n=1$ we have $$\prod_{1 \leq i \leq n} (1 - g^{e_i})^k \cdot \prod_{1 \leq i \leq n} (1 - g^{a_i})^k = (1 - g^{e_1})^k \cdot (1 - g^{a_1})^k,$$
which differs from $(1 - g^{e_1})^{2k}$ by an invertible element, and $(1 - g^{e_1})^{2k} \neq 0 \in \mathbb{F}_p[\mathbb{F}_p]$, since
$2k < p$.

Let us assume in the following that $n \geq 2$ and the statement holds for $\ell < n$. 

\noindent
For the sake of contradiction let us assume that $M \in M^{n \times n}(\mathbb{F}_p)$ is a counterexample, that is, the identity 
$$ \prod_{1 \leq i \leq n} (1 - g^{e_i})^k \cdot \prod_{1 \leq i \leq n} (1 - g^{a_i})^k = 0 \in \mathbb{F}_p[\mathbb{F}_p^n]$$
holds.

We will prove that we can delete one of the terms $(1 - g^{a_i})^k$ in such a way that the equation still holds.

\begin{lemma}\label{onemiss}
There exists an index $j \in [n]$ such that  
$$\prod_{1 \leq i \leq n} (1 - g^{e_i})^k \cdot \prod_{1 \leq i \leq n, i \neq j} (1 - g^{a_i})^k = 0 \in \mathbb{F}_p[\mathbb{F}_p^n].$$
\end{lemma}
\begin{proof}

Let us define the set $S\subseteq \mathbb{F}_p^n$ as
$$S: = \left\{\sum_{1 \leq i \leq n} y_i \cdot e_i\ \Big|\ y_i \in A_i ,\ 1 \leq i \leq n\right\}.$$

For a vector $v \in S$ let us write $v = \sum\limits_{1 \leq i \leq n} z_i \cdot a_i $ and let $I_v \subseteq [n]$ be the subset such that $x_i \in  B_i$ if and only if $i \in I_v$.

Note that $I_v \neq [n]$ for every vector $v \in S$. Let us take a vector $u \in S$ such that $I_u$ is maximal with respect to inclusion and let us choose an index $1 \leq j \leq n$ such that $j \neq I_u$.

Let us write 
$$u = \sum_{1 \leq i \leq n} z_{u, i }\cdot a_i  = \sum_{1 \leq i \leq n} y_{u, i} \cdot e_i.$$

We know that $y_{u, i }\in A_i$ if $1 \leq i \leq n$, which means, using the fact that the set $A_i$ is $S_k$-type, that there exists an element $\lambda_i \in \mathbb{F}_p^*$ such that $y_{u, i} + \ell \lambda_i \in A_i$ for every $-k \leq \ell \leq k $.

Similarly, if $i \in I_u$, then there exists $\mu_i \in \mathbb{F}_p^*$ such that $z_{u, i} + \ell \mu_i \in B_{i}$ for every $ -k \leq \ell \leq k$.
If $i \notin I_u$ but $i \neq j$ then let us choose an arbitrary $\mu_i \in \mathbb{F}_p^*$.

For $i = j$ we may choose a residue $\mu_j \in \mathbb{F}_p^*$ such that $z_{u, j} + \ell \mu_j \in A_{t_j}$ for $ 1 \leq \ell \leq k$, since the set $B_j$ is $N_k$-type.

We know that 
$$\prod_{1 \leq i \leq n} (1 - g^{e_i})^k \cdot \prod_{1 \leq i \leq n} (1 - g^{a_i})^k = 0 \in \mathbb{F}_p[\mathbb{F}_p^n],$$
so let us apply Lemma~\ref{testmodp} for the case $t_i = t'_i = k\ (i\in [n])$ and $\lambda_1, \dots, \lambda_n, \mu_1, \dots, \mu_n \in \mathbb{F}_p^*$.

We get that for all integers $0 \leq k_i \leq k,\ 0 \leq k'_i \leq k$, if we consider the functions $f: \mathbb{F}_p^n \to \mathbb{F}_p$ given by the formula $$f(\sum\limits_{1 \leq i \leq n} (\ell_i - k_i)  \lambda_i e_i) =   
(-1)^{\sum_{1 \leq i \leq n} \ell_i} \cdot \prod_{1 \leq i \leq n} {\binom{k}{\ell_i}},$$ 
where $0 \leq \ell_i \leq k,\ i\in [n]$ and $f$ takes values $0$ otherwise and similarly $g: \mathbb{F}_p^n \to \mathbb{F}_p$ is given by the formula 
$$g(\sum\limits_{1 \leq i \leq n} (k'_i - \ell'_i)  \mu_i a_i) =   
(-1)^{\sum_{1 \leq i \leq n} \ell'_i} \cdot \prod_{1 \leq i \leq n} {\binom{k}{\ell'_i}},$$
where $0 \leq \ell'_i \leq k,\ i\in [n]$, then we have $\langle f, g \rangle = 0$.

Restricting to the case $k'_j = k$ and shifting the functions $f, g$ by the vector $u$ we get that for all integers $0 \leq k_i \leq k,\ 0 \leq k'_i \leq k$ such that $k'_j = k$, if we take the function 
$f: \mathbb{F}_p^n \to \mathbb{F}_p$ given by the formula 
$$f(u + \sum\limits_{1 \leq i \leq n} (\ell_i - k_i)  \lambda_i e_i) =   
(-1)^{\sum_{1 \leq i \leq n} \ell_i} \cdot \prod_{1 \leq i \leq n} {\binom{k} { \ell_i}},$$ where $0 \leq \ell_i \leq k, i\in [n]$
and $f$ takes values $0$ otherwise and similarly $g: \mathbb{F}_p^n \to \mathbb{F}_p$ is given by the formula $$g(u + \sum\limits_{1 \leq i \leq n} (k'_i - \ell'_i)  \mu_i a_i) =   
(-1)^{\sum_{1 \leq i \leq n} \ell'_i} \cdot \prod_{1 \leq i \leq n} {\binom{k}{ \ell'_i}},$$ 
where $0 \leq \ell'_i \leq k,\ i\in [n]$, then we have $\langle f, g \rangle = 0$.

We claim that if $0 \leq \ell_i,\ell'_i \leq k$ (for $i\in[n]$) such that $\ell'_j \neq k$,
then 
$$u + \sum_{1 \leq i \leq n} (\ell_i - k_i)  \lambda_i e_i \neq u + \sum_{1 \leq i \leq n} (k'_i - \ell'_i)  \mu_i a_i.$$

Indeed, we know that 
$$u + \sum_{1 \leq i \leq n} (\ell_i - k_i)  \lambda_i e_i = \sum_{1 \leq i \leq n} (y_{u, i} + (\ell_i - k_i)  \lambda_i ) e_i $$
and
$y_{u, i} + (\ell_i - k_i)  \lambda_i \in A_i$, since the set $A_i$ is $S_k$-type.

On the other hand 
$$u + \sum_{1 \leq i \leq n} (k'_i - \ell'_i)  \mu_i a_i = \sum_{1 \leq i \leq n} (z_{u, i} + (k'_i - \ell'_i)  \mu_i ) a_i.$$

Now notice that if $i \in I_u$, then the coefficient of $a_i$ in $u + \sum\limits_{1 \leq i \leq n} (k'_i - \ell'_i)  \mu_i a_i $ is in $B_i$ by the construction of the residues $\mu_i$ and if $i = j$,
then $z_{u, j} + (k'_j - \ell'_j)  \mu_i \in B_{j}$, since $-k \leq k'_j - \ell'_j \leq k$.

This indeed means that 
$$u + \sum_{1 \leq i \leq n} (\ell_i - k_i)  \lambda_i e_i = u + \sum_{1 \leq i \leq n} (k'_i - \ell'_i)  \mu_i a_i$$
is not possible, since if they were equal to the same vector $w$, 
then $w \in S$ and $I_w$ would properly contain $I_u$  which contradicts the maximality of $I_u$.

\noindent
Hence, we obtained the following statement:

For all integers $0 \leq k_i,k'_i \leq k\ (i\in [n]\setminus\{j\})$, 
if  the function $f: \mathbb{F}_p^n \to \mathbb{F}_p$ is given by the formula $$f(\sum\limits_{1 \leq i \leq n} (\ell_i - k_i)  \lambda_i e_i) =   
(-1)^{\sum_{1 \leq i \leq n} \ell_i} \cdot \prod_{1 \leq i \leq n} {\binom{k} { \ell_i}},$$
where $0 \leq \ell_i \leq k\ (i\in [n])$ and $f$ takes values $0$ otherwise and similarly $g: \mathbb{F}_p^n \to \mathbb{F}_p$ is given by the formula $$g(\sum\limits_{1 \leq i \leq n, i \neq j} (k'_i - \ell'_i)  \mu_i a_i) =   
(-1)^{\sum_{1 \leq i \leq n} \ell'_i} \cdot \prod_{1 \leq i \leq n} {\binom{k} { \ell'_i}}$$
where $0 \leq \ell'_i \leq k$ (for $i\in [n]\setminus\{j\}$), then we have $\langle f, g \rangle = 0$.

By Lemma~\ref{testmodp} this yields  that 
$$\prod_{1 \leq i \leq n} (1 - g^{e_i})^k \cdot \prod_{\substack{1 \leq i \leq n,\\ i \neq j}} (1 - g^{a_i})^k = 0 \in \mathbb{F}_p[\mathbb{F}_p^n]$$
which proves our lemma completely.
\end{proof}

Now, we apply Lemma~\ref{onemiss} to finish the proof of our proposition.

\noindent
By Lemma~\ref{onemiss} we know that there is an index $1 \leq j \leq n$ such that  $$\prod_{1 \leq i \leq n} (1 - g^{e_i})^k \cdot \prod_{1 \leq i \leq n, i \neq j} (1 - g^{a_i})^k = 0 \in \mathbb{F}_p[\mathbb{F}_p^n].$$

We know that the matrix $M$ is invertible so we get that there is an index $1 \leq \ell \leq n$ such that the submatrix $M_{\ell, j}$ is invertible.

Let us write $a_i = a'_i + a_{i, \ell} e_{\ell}$, then let the subgroup $G$ be defined as 
$$G:=\langle a'_1, \dots, a'_{j-1}, a'_{j+1}, \dots, a'_n \rangle = \langle e_1, \dots, e_{\ell-1}, e_{\ell+1}, \dots, e_n \rangle  \cong \mathbb{F}_p^{n-1}.$$

For each $1 \leq i \leq n,\ i \neq \ell$ we have that $1 - g^{a_i} = (1-g^{a'_i}) + (1- g^{e_{\ell}}) \cdot y_i$ for some element $y_i \in \mathbb{F}_p[\mathbb{F}_p^n]$.

Notice that we have 
$$B = \prod_{1 \leq i \leq n} (1 - g^{e_i})^k \cdot   \prod_{1 \leq i \leq n, i \neq j} ((1-g^{a'_i}) + (1- g^{e_\ell}) \cdot y_i)^k  = 0 \in \mathbb{F}_p[\mathbb{F}_p^n].$$

We may write 
$$0 = B = \sum_{k \leq r \leq p-1} b_r (1 - g^{e_l})^r,$$ where $b_k \in \mathbb{F}_p[G].$
On the other hand we have 
$$b_k = \prod_{\substack{1 \leq i \leq n,\\ i \neq \ell}} (1 - g^{e_i})^k \cdot   \prod_{\substack{1 \leq i \leq n,\\ i \neq j}} (1 - g^{a'_i})^k \in \mathbb{F}_p[G],$$
so we obtain that
\begin{equation}
\prod_{\substack{1 \leq i \leq n,\\ i \neq \ell}} (1 - g^{e_i})^k \cdot   \prod_{\substack{1 \leq i \leq n,\\ i \neq j}} (1 - g^{a'_i})^k = 0 \in \mathbb{F}_p[G].
\end{equation}

However, this means that property (P4), i.e. the modulo $p$ group ring identity, holds for the invertible matrix $M_{k, j} \in M^{(n-1) \times (n-1)}$, which contradicts the induction hypothesis. This finishes the proof of part (i).

Now, part (ii) and part (iii) follow form Corollary~\ref{goodsubsets}.

\end{proof}

\section{The case of small primes and connections of group ring identities}\label{sec-small-primes}

In this section in the case of small primes we show the importance of the equivalence of the two group ring identities asked in Conjecture~\ref{41} and we show how the Alon-Jaeger-Tarsi conjecture would follow from it.

Now, we show how Conjecture~\ref{41} would imply the Alon-Jaeger-Tarsi conjecture:
\begin{theoremC}\label{impliesAT}
For every prime $p>3$ Conjecture~\ref{41} implies the Alon-Jaeger-Tarsi conjecture.
\end{theoremC}

\begin{proof}

For the sake of contradiction, let us assume that Conjecture~\ref{41} holds, but the Alon-Jaeger-Tarsi conjecture does not hold for a prime $p$.

Let $n$ be the smallest dimension such that there is a counterexample to it, and let $M \in M^{n \times n}(\mathbb{F}_p)$ be a counterexample with row vectors $a_i \ (i\in [n])$. Let us write $G := \mathbb{F}_p^n$.

From Proposition~\ref{properties} we know that the $\mathbb{Z}$ group ring identity $$\prod_{1 \leq j \leq n} (1 - g^{e_j}) \cdot   \prod_{1 \leq j \leq n} (1 - g^{a_j})  = 0 \in \mathbb{Z}[G]$$
holds and consequently the mod $p$ group ring identity 
$$\prod_{1 \leq j \leq n} (1 - g^{e_j}) \cdot   \prod_{1 \leq j \leq n} (1 - g^{a_j})  = 0 \in \mathbb{F}_p[G]$$
also holds.

Let us take an arbitrary index $ i \in [n]$, then first we prove the following lemma:

\begin{lemma}
We have 
$$\prod_{\substack{1 \leq j \leq n,\\ j \neq i}} (1 - g^{e_j}) \cdot   \prod_{1 \leq j \leq n} (1 - g^{a_j})  = 0 \in \mathbb{Z}[G].$$
\end{lemma}
\begin{proof}

Let us write $a_j = a'_j + a_{j, i} e_i$, then let us consider the subgroup $G':=\langle a'_1, \dots, a'_n \rangle = \langle e_1, \dots, e_{i-1}, e_{i+1}, \dots, e_n\rangle$.

For each $j\in [n]$ we have that $1 - g^{a_j} = (1-g^{a'_j}) + (1- g^{e_i}) \cdot y_j$ for some element $y_j \in \mathbb{Z}[G]$.

Observe that we have
\begin{equation*}
B := \prod_{1 \leq j \leq n} (1 - g^{e_j}) \cdot   \prod_{1 \leq j \leq n} ((1-g^{a'_j}) + (1- g^{e_i}) \cdot y_j)  = 0 \in \mathbb{F}_p[G].
\end{equation*}
\noindent
We may write 
$$0 = B = \sum_{1 \leq k \leq p-1} b_k (1 - g^{e_i})^k,$$
where $b_k \in \mathbb{F}_p[G']$.

On the other hand we have 
$$b_1 = \prod_{\substack{1 \leq j \leq n,\\ j \neq i}} (1 - g^{e_j}) \cdot   \prod_{1 \leq j \leq n} (1 - g^{a'_j}) \in \mathbb{F}_p[G'],$$
so we get that
\begin{equation*}
\prod_{\substack{1 \leq j \leq n,\\ j \neq i}} (1 - g^{e_j}) \cdot   \prod_{1 \leq j \leq n} (1 - g^{a'_j}) = 0 \in \mathbb{F}_p[G'].
\end{equation*}

Now, we claim that we also have
\begin{equation*}
 \prod_{\substack{1 \leq j \leq n,\\ j \neq i}} (1 - g^{e_j}) \cdot   \prod_{1 \leq j \leq n} (1 - g^{a'_j}) = 0 \in \mathbb{Z}[G'].   
\end{equation*}

Notice that this is a similar implication as the statement of Conjecture~\ref{41}, but we %cannot imply 
do not get it directly since the vectors $a'_j$ do not form a basis of the subgroup $G' \cong \mathbb{F}_p^{n-1}$, indeed there is one more vector.

Assume to the contrary that 
$$\prod_{\substack{1 \leq j \leq n,\\ j \neq i}} (1 - g^{e_j}) \cdot   \prod_{1 \leq j \leq n} (1 - g^{a'_j}) \neq 0 \in \mathbb{Z}[G'],$$
we will construct a counterexample to Conjecture~\ref{41} from this and from our original counterexample $M$ to the Alon-Jaeger-Tarsi conjecture.

\smallskip

We may assume (by reindexing, if necessary) that $\langle a'_2, \dots, a'_n \rangle = G'$.

Let us consider the mod $p$ vector space with the following basis vectors: 
$$V' = \langle e_1, \dots, e_n , e_{j, k}\ |\  1 \leq j \leq 2n,\ 1 \leq k \leq n,\  k \neq i \rangle,$$ so 
$V' \cong \mathbb{F}_p^{2n^2 - n}$.

Now, we will construct two  bases $B_1$ and $B_2$ of $V'$.

\noindent
Let us give $B_1$ in the following form as a disjoint union:
$$B_1 = B'_1 \cup \bigcup\limits_{1 \leq t \leq 2n} B'_{1, t},$$
where $B'_1= \{e_j + \sum\limits_{1 \leq k \leq n, k \neq i} a'_{1, k} e_{j, k}\ |\ j \in [n] \}$ and $B'_{1, t}=\{e_{t, k}\ |\ k\in[n],\ k \neq i\}$.

\medskip

\noindent
Similarly, let us give $B_2$ in the form $B_2 = B'_2 \cup \bigcup\limits_{1 \leq t \leq 2n} B'_{2, t}$,
where
$$B'_2= \{ \sum\limits_{1 \leq k \leq n} a_{j-n, k} e_{k} + \sum\limits_{\substack{1 \leq k \leq n,\\ k \neq i}} a'_{1, k} e_{j, k}\ |\ j \in [n+1, 2n] \},$$ 
$$B'_{2, t} = \{\sum\limits_{\substack{1 \leq k \leq n, \\ k \neq i}} a'_{\ell, k} e_{t, k}\ |\ \ell \in [2, n] \}.$$

\noindent
It can be easily checked that  $B_1$ and $B_2$ are two bases of $V'$:

\noindent
Let us write $v_j: = \sum\limits_{1 \leq k \leq n, k \neq i} a'_{1, k} e_{j, k}$ if $j \in [2n]$
and $w_j := \sum\limits_{1 \leq k \leq n} a_{j, k} e_{k}$ for $j \in [n]$.

First, we claim that 
$$\prod_{b \in B_1} (1-g^b) \cdot \prod_{b \in B_2} (1-g^b) = 0 \in \mathbb{F}_p[V'].$$

\noindent
For an index $t \in [2n]$ let us denote 
$$U_t := \prod\limits_{b \in B'_{1, t}} (1-g^b) \cdot \prod\limits_{b \in B'_{2, t}} (1-g^b).$$

\noindent
Notice that by the properties of the vectors $a'_1, \ldots, a'_n, e_1, \ldots, e_{i-1}, e_{i+1}, \ldots, e_n$
we know that for each $t \in [2n]$ we have $U_t \cdot (1 - g^{v_t})  = 0 \in \mathbb{F}_p[V']$.

Observe that for $j\in [n]$ we can write 
$$1 - g^{e_j + v_j} = (1 - g^{e_j}) + (1- g^{v_j}) z_j,$$
where $z_j \in \mathbb{F}_p[V']$.

Similarly, for $j \in [n]$ we can write $$1 - g^{w_j + v_{j+n}} = (1 - g^{w_j}) + (1- g^{v_{j+n}}) r_j,$$
where $r_j \in \mathbb{F}_p[V']$.

It follows from these equations that
\begin{equation*}
 \prod_{b \in B_1} (1-g^b) \cdot \prod_{b \in B_2} (1-g^b) = \prod_{1 \leq t \leq 2n} U_t \cdot \prod_{1 \leq j \leq n}(1 - g^{e_j}) \cdot \prod_{1 \leq j \leq n}(1 - g^{w_j}) \in \mathbb{F}_p[V'].
\end{equation*}

On the other hand, since $M$ is a counterexample to the Alon-Jaeger-Tarsi conjecture we know that 
$$\prod_{1 \leq j \leq n}(1 - g^{e_j}) \cdot \prod_{1 \leq j \leq n}(1 - g^{w_j}) = 0 \in \mathbb{F}_p[V'], $$
which indeed yields 
$$\prod_{b \in B_1} (1-g^b) \cdot \prod_{b \in B_2} (1-g^b) = 0 \in \mathbb{F}_p[V'].$$

\emph{Next, we will prove that on the other hand we have 
$$\prod_{b \in B_1} (1-g^b) \cdot \prod_{b \in B_2} (1-g^b) \neq 0 \in \mathbb{Z}[V'].$$}

\smallskip

For this we have to show that there is a vector $x \in V'$ such that $\langle x, b\rangle \neq 0$ if $b \in B_1$
or $b \in B_2$, where $\langle \cdot,\cdot \rangle$ denotes the standard scalar product corresponding to the basis $e_1, \dots, e_n , e_{j, k}$.

Let us fix the values of $x_1, \dots, x_n$ in an arbitrary way, we show that we can extend these
with values of $x_{j, k}$ in a proper way such that the conditions hold.

Indeed for each $j \in [2n]$ we get the conditions for the values $x_{j, k}$ that
$x_{j, k} \neq 0$ for each $k \in [n],\ k \neq i$ and $\sum\limits_{1 \leq k \leq n, k \neq i} a'_{\ell, k} x_{j, k} \neq 0$ for each $2 \leq \ell \leq n$ and $\sum\limits_{1 \leq k \leq n,\ k \neq i} a'_{1, k} x_{j, k} \neq c_j$
for some $c_j \in \mathbb{F}_p$.

Now from our assumptions on the vectors $a'_1, \dots, a'_n, e_1, \dots, e_{i-1}, e_{i+1}, \dots, e_n$
we can find a vector $y \in \langle e_{j, k}\ | \  1 \leq k \leq n,\ k \neq i\rangle$ such that the first two types of conditions hold for it and 
$$\sum\limits_{1 \leq k \leq n, k \neq i} a'_{1, k} y_{j, k} \neq 0.$$ 

If $c_j = 0$, then we can choose $x_{j, k} = y_{j, k}$ and if $c_j \neq 0$, then we can choose $x_{j, k} = \lambda \cdot y_{j, k}$ with some appropriate $0 \neq \lambda \in \mathbb{F}_p$.

Thus we have
$$\prod\limits_{b \in B_1} (1-g^b) \cdot \prod\limits_{b \in B_2} (1-g^b) \neq 0 \in \mathbb{Z}[V'],$$
which contradicts our assumption that Conjecture~\ref{41} holds.

\medskip

We proved that 
$$\prod_{\substack{1 \leq j \leq n,\\ j \neq i}} (1 - g^{e_j}) \cdot   \prod_{1 \leq j \leq n} (1 - g^{a'_j}) = 0 \in \mathbb{Z}[G'],$$
let us define 
$$Q := \prod_{\substack{1 \leq j \leq n,\\ j \neq i}} (1 - g^{e_j}) \cdot   \prod_{1 \leq j \leq n} (1 - g^{a_j}).$$ 

This means however that $Q \in \langle 1 - g^{e_i} \rangle \subseteq \mathbb{Z}[G]$, where $\langle 1 - g^{e_i}\rangle $ is the principal ideal generated by the element $1 - g^{e_i} \in \mathbb{Z}[G]$, let us have $Q = q \cdot (1 - g^{e_i})$.

Now, notice that we know that $Q \cdot (1 - g^{e_i}) = 0$, which means that $q \cdot (1 - g^{e_i})^2 = 0$.

However, in $\mathbb{Z}[G]$ the kernel of the multiplication with the element $(1 - g^{e_i})^2$ is the same as 
the kernel of the multiplication with the element $(1 - g^{e_i})$.

Indeed we get that $F(i(q)) \cdot F(i((1 - g^{e_i})))^2 = 0$, from which we obtain that 
$F(i(q)) \cdot F(i((1 - g^{e_i}))) = 0$.
By applying the inverse Fourier inversion we get indeed that $Q = q \cdot (1 - g^{e_i}) = 0$ which
proves our lemma completely.
\end{proof}

\bigskip

Now, let us return to the proof of Theorem C, we know that
\begin{equation}\label{eq1}
\prod_{1 \leq j \leq n} (1 - g^{e_j}) \cdot   \prod_{1 \leq j \leq n} (1 - g^{a_j})  = 0 \in \mathbb{F}_p[G].
\end{equation}

\noindent
For an element $ x = \sum\limits_{ v \in  \mathbb{F}_p^n} x_v g^v \in \mathbb{F}_p[G]$ and an index $i \in [n]$, let us define the operation 
$$\partial_i(x) := \sum\limits_{ v \in  \mathbb{F}_p^n} x_v  \cdot v_i \cdot g^v \in \mathbb{F}_p[G].$$

\noindent
It is easy to check that $\partial_i$ satisfies the property $\partial_i (x \cdot y) = \partial_i(x) \cdot y + x \cdot \partial_i(y)$ if $x, y \in  \mathbb{F}_p[G]$.

\noindent
If we apply the partial derivative $\partial_i$ to \eqref{eq1}, we get
\begin{equation*}
   \prod_{\substack{1 \leq j \leq n ,\\ j \neq i} } (1 - g^{e_j}) \cdot \prod_{1 \leq j \leq n} (1 - g^{a_j})  + \sum_{1 \leq k \leq n} a_{k, i} \cdot \prod_{1 \leq j \leq n} (1 - g^{e_j}) \cdot   \prod_{\substack{1 \leq j \leq n,\\ j \neq k}} (1 - g^{a_j}) = 0 \in \mathbb{F}_p[G].
\end{equation*}

On the other hand by our main lemma %%%% írni h melyik lemma
we know that $$\prod_{\substack{1 \leq j \leq n ,\\ j \neq i} } (1 - g^{e_j}) \cdot \prod_{1 \leq j \leq n} (1 - g^{a_j}) = 0 \in \mathbb{F}_p[G],$$
so we get for every $i\in [n]$ that
\begin{equation*}
 \sum_{1 \leq k \leq n} a_{k, i} \cdot \prod_{1 \leq j \leq n} (1 - g^{e_j}) \cdot   \prod_{\substack{1 \leq j \leq n,\\ j \neq k}} (1 - g^{a_j}) = 0 \in \mathbb{F}_p[G].
\end{equation*}

\noindent
We know that the matrix $M$ is invertible so we obtain that
\begin{equation*}
 \prod_{1 \leq j \leq n} (1 - g^{e_j}) \cdot   \prod_{2 \leq j \leq n} (1 - g^{a_j}) = 0 \in \mathbb{F}_p[G].
\end{equation*}

We may assume (by reindexing, if necessary)  that the determinant of the matrix obtained from $M$ by deleting the first row and first column is nonzero.

\noindent
Let us write $a_j = a_{j, 1} e_1 + a'_j$, then $a'_2, \ldots, a'_n$ is a basis of $G' = \langle e_2, \dots, e_n\rangle $.

\noindent
Since for $j \in [2, n]$ we have $1 - g^{a_j} = (1 - g^{a'_j}) + (1-g^{e_1}) y_j$, we get that
\begin{equation*}
z = \prod_{1 \leq j \leq n} (1 - g^{e_j}) \cdot   \prod_{\substack{1 \leq j \leq n,\\ j \neq 1}} (1 - g^{a'_j}) + (1-g^{e_1}) y_j = 0 \in \mathbb{F}_p[G].
\end{equation*}
We may write $0 = z = \sum\limits_{1 \leq k \leq p-1} (1 - g^{e_1})^k z_k$, where $z_k = 0 \in  \mathbb{F}_p[G']$.

\noindent
On the other hand we have 
$$z_1 = \prod\limits_{2 \leq j \leq n} (1 - g^{e_j}) \cdot   \prod\limits_{2 \leq j \leq n} (1 - g^{a'_j})  = 0 \in \mathbb{F}_p[G'].$$

\noindent
Hence, finally, we obtain that
\begin{equation*}
\prod_{2 \leq j \leq n} (1 - g^{e_j}) \cdot   \prod_{2 \leq j \leq n} (1 - g^{a'_j})  = 0 \in \mathbb{F}_p[G'].
\end{equation*}

However, using again the statement of Conjecture~\ref{41} we get this way a counterexample to the Alon-Jaeger-Tarsi conjecture with smaller dimension than $n$ which is a contradiction, since we took a minimal counterexample.

This contradiction finishes the proof of Theorem C.
\end{proof}

\bigskip

Let us prove finally a proposition which expresses Conjecture~\ref{41} in another form:

\begin{proposition}\label{symmetric}
Conjecture~\ref{41} is equivalent to the statement that if $n\geq 1$, then there exists no nonsingular matrix $M \in M^{n \times n}(\mathbb{F}_p)$ such that if $e_i \in \mathbb{F}_p^n\ ,i \in [n]$ are the standard basis vectors and $a_i \in \mathbb{F}_p^n\ ,i \in [n]$ are the row vectors of $M$, then $$\sigma_{2n-i}((1 - g^{e_1}), \dots, (1 - g^{e_n}), (1 - g^{a_1}), \dots, (1 - g^{a_n}) ) = 0 \in \mathbb{F}_p[\mathbb{F}_p^n] $$ for every $0 \leq i \leq p-1$, where $\sigma_{2n-i}$ are the elementary symmetric polynomials.
\end{proposition}
\begin{proof}

For brevity let us write $\sigma_{2n-i}((1 - g^{e_k}), (1 - g^{a_\ell}))$ for 
$$\sigma_{2n-i}((1 - g^{e_1}), \dots, (1 - g^{e_n}), (1 - g^{a_1}), \dots, (1 - g^{a_n}) ).$$

Let us assume first that there is a matrix $M \in M^{n \times n}(\mathbb{F}_p)$ such that $\sigma_{2n-i}((1 - g^{e_k}), (1 - g^{a_\ell})) = 0$ for every $0 \leq i \leq p-1$.

Let us consider the vector space $ V = \langle e_1, \dots, e_n, f\rangle  \cong\mathbb{F}_p^{n+1} $, where $f$ is a new basis vector.

\noindent
First, we claim that 
$$\prod_{1 \leq j \leq n} (1 - g^{f + e_j}) \cdot   \prod_{1 \leq j \leq n} (1 - g^{f + a_j})  = 0 \in \mathbb{F}_p[V].$$

\noindent
Indeed, we have $1 - g^{f + e_j} = (1-g^f) + g^f(1 - g^{e_j})$ and $(1-g^f)^{p} = 0$, so we obtain that
\begin{multline*}
\prod_{1 \leq j \leq n} (1 - g^{f + e_j}) \cdot   \prod_{1 \leq j \leq n} (1 - g^{f + a_j})= \\
=
\sum_{0\leq i \leq p-1} (1-g^f)^i \cdot g^{(2n - i)f} \cdot \sigma_{2n-i}((1 - g^{e_k}), (1 - g^{a_l})) = 0 \in \mathbb{F}_p[V].
\end{multline*}

Notice that the vectors $f + e_j$ are linearly independent, let us extend this system to a basis of $V$: 
$$b_1 = f + e_1 ,\  \dots,\ b_n = f + e_n ,\ b_{n+1} = f + v,$$
where $v \in \langle e_1, \dots, e_n \rangle$.

Similarly, let us take another basis of $V$: $$b'_1 = f + a_1 ,\ \dots,\ b'_n = f + a_n ,\ b'_{n+1} = f + v',$$
where $v' \in \langle e_1, \dots, e_n \rangle $.

Note that we have 
$$\prod_{1 \leq j \leq n+1} (1 - g^{b_j}) \cdot   \prod_{1 \leq j \leq n+1} (1 - g^{b'_j})= 0 \in \mathbb{F}_p[V].$$

On the other hand 
$$\prod_{1 \leq j \leq n+1} (1 - g^{b_j}) \cdot   \prod_{1 \leq j \leq n+1} (1 - g^{b'_j}) \neq 0 \in \mathbb{Z}[V],$$ since for the standard scalar product $\langle \cdot,\cdot\rangle $ for the basis $ e_1, \ldots, e_n, f$
one has $\langle f, b_j\rangle \neq 0$ for all $j\in [n] $ and $\langle f, b'_j \rangle \neq 0$ for all $j\in [n] $
which contradicts Conjecture~\ref{41}.

Let us assume now that Conjecture~\ref{41} is false, so there is a vector space $V \cong \mathbb{F}_p^n$
and two bases $b_1, \ldots, b_n \in V$ and $b'_1, \ldots, b'_n \in V$ such that:

\begin{equation*}
\prod_{1 \leq j \leq n+1} (1 - g^{b_j}) \cdot \prod_{1 \leq j \leq n+1} (1 - g^{b'_j})= 0 \in \mathbb{F}_p[V],
\end{equation*}
but there is a linear function $f: V \to \mathbb{F}_p$ such that $f(b_j) \neq 0\ (j\in [n])$ and $f(b'_j) \neq 0\ (j\in [n])$.

Observe that the identity 
$$\prod_{1 \leq j \leq n+1} (1 - g^{b_j}) \cdot \prod_{1 \leq j \leq n+1} (1 - g^{b'_j})= 0 \in \mathbb{F}_p[V]$$
remains true if we dilate the basis vectors $b_1, \ldots, b_n, b'_1, \ldots, b'_n $,
so we may assume that there is a $1$-codimensional subspace $V' \subseteq V$ and a vector $f \notin V'$
such that $b_1, \ldots, b_n, b'_1, \ldots, b'_n \in f + V'$.

%Notice that if we introduce a formal free variable $t$ with the condition that $t^p = 1$, then we get the identity
%\begin{equation}\label{tidentity}
%\prod_{1 \leq j \leq n+1} (1 - t \cdot g^{b_j}) \cdot \prod_{1 \leq j \leq n+1} (1 - t \cdot g^{b'_j})= 0 \in \mathbb{F}_p[V, t].
%\end{equation}
%However, this means that the identity should also hold in the $t$-quantified form.
Now, let us consider a vector space $V'': = \langle b_1, \dots, b_n, w \rangle$, where $b_1, \ldots, b_n, w$ is a basis of $V''$, we claim  that
\begin{equation}\label{gwquantified}
\prod_{1 \leq j \leq n+1} (1 - g^{b_j} \cdot g^w) \cdot \prod_{1 \leq j \leq n+1} (1 - g^{b'_j} \cdot g^w)= 0 \in \mathbb{F}_p[V^*].
\end{equation}

Indeed, different products of elements in the product 
\begin{equation}\label{eq-gbj}
\prod_{1 \leq j \leq n+1} (1 - g^{b_j}) \cdot \prod_{1 \leq j \leq n+1} (1 - g^{b'_j})
\end{equation}
can be the same if we choose the same amount of $g^{b_j}, g^{b'_j}$ form the brackets modulo $p$ because of the condition that $b_1, \ldots, b_n, b'_1, \ldots, b'_n \in f + V'$.

This indeed means that \eqref{gwquantified} should hold, since on the left hand side if we expand the brackets then the terms that coincide are exactly the same as the terms that coincide in \eqref{eq-gbj}.

Therefore, we get that
\begin{equation*}
\prod_{1 \leq j \leq n+1} (1 - g^{b_j + w}) \cdot \prod_{1 \leq j \leq n+1} (1 - g^{b'_j + w})= 0 \in \mathbb{F}_p[V^*].
\end{equation*}

Now we have $1 - g^{w + b_j} = (1-g^w) + g^w(1 - g^{b_j})$ and $(1-g^w)^{p} = 0$, so we obtain that
\begin{multline*}
\prod_{1 \leq j \leq n+1} (1 - g^{b_j + w}) \cdot \prod_{1 \leq j \leq n+1} (1 -  g^{b'_j + w})= \\
=
\sum_{0\leq i \leq p-1} (1-g^w)^i \cdot g^{(2n - i)w} \cdot \sigma_{2n-i}((1 - g^{b_k}), (1 - g^{b'_\ell})) = 0 \in \mathbb{F}_p[V''].
\end{multline*}

It is easy to see that the elements $(1-g^w)^i \cdot g^{(2n - i)w}\ (0 \leq i \leq p-1)$ form a basis of
$\mathbb{F}_p[\langle w\rangle]$, so this yields that $\sigma_{2n-i}((1 - g^{b_k}), (1 - g^{b'_\ell})) = 0$ for $0 \leq i \leq p-1$,
which finishes the proof of the  proposition.
\end{proof}

We see from the preceding proposition that proving Conjecture~\ref{41} might be easier than  proving that the mod $p$ group ring identity 
$$\prod_{1 \leq j \leq n+1} (1 - g^{b_j}) \cdot \prod_{1 \leq j \leq n+1} (1 - g^{b'_j})= 0 \in \mathbb{F}_p[V]$$
cannot hold for two bases of a vector space $V$, indeed, we get $p-1$ more equations.

\section{Concluding remarks and open questions}\label{sec-open}

In this paper we proved that the Alon-Jaeger-Tarsi conjecture holds for primes $61<p\ne 79$. More generally, it was shown that for any $k\geq 1$ it is possible to find a vector $x\in\mathbb{F}_p^n$ such that all entries of $x$ and $Mx$ avoid $k$ chosen residues (which can be chosen independently for each entry), assuming that $p$ is sufficiently large (compared to $k$). For our methods to work we need that the inequality $S_k(p)N_k(p)<p-1$ should hold for the prime $p$.

From the proof of Lemma~\ref{Skepsilon} it is easy to see  that for a fixed $k \geq 2$ and $\varepsilon > 0$ 
$S_k(p) = o(p^{\varepsilon})$, which bound is still much worse than the magnitude $\log_2(p)$ in the case $k=1$. 

\noindent
We pose the following two questions:
\begin{question}
What is the order of magnitude of $S_k(p)$ and $N_k(p)$ in general for fixed $k$?
\end{question}

\begin{question}
Let $k \geq 2$ be an integer. What is the order of magnitude of the smallest prime $p$ such that $p-1 > N_k(p) \cdot S_k(p)$?
\end{question}

We have seen that the key idea of proving the main results in Theorem~\ref{mainth} was to use some kind of duality.
More specially we use the falseness of a statement which is similar to DeVos' conjecture in the dual setup in an extremal case when it is false tautologically by a simple counting argument.

On the other hand this duality argument fails at the first glance in the case when we have more than two bases.
We pose the following conjecture about this situation (for a proof of a weaker form of this kind of statement see \cite{Alweiss}):
\begin{conjecture}\label{conj-AJTmorebase}
Let $k \geq 3$ be an integer. Then there exists a threshold $C(k)$ such that if $p > C(k)$, and we have
$k$ nonsingular matrices $M_1, \dots, M_{k}$ over $\mathbb{F}_p$, then there is a vector $x$ such that each of the vectors $M_ix$ ($1\leq i\leq k$) have only nonzero entries.
\end{conjecture}

\newpage{}

\section*{Appendix}

Our methods imply that Conjecture~\ref{conj-AJT} holds for a prime $p$ if $S_1(p)<\sqrt{p-1}$.
With the help of Lemma~\ref{SN1} it can be seen that $S_1(p)<\sqrt{p-1}$ if $199\leq p\ne 257$ is a prime. By using computer calculations we could check that in fact $S_1(p)<\sqrt{p-1}$  holds for a prime $p$ if and only if $61<p\neq 79$. In the table below we give constructions showing that $S_1(p)<\sqrt{p-1}$ also holds for the primes lying in the set $(61,199)\cup\{257\}\setminus \{79\}$, which cases are not yet covered by Lemma~\ref{SN1}. Up to 151 a brute force exhaustive search gave the optimal constructions.  From 157 we do not know whether the listed constructions are optimal. These were found by using some heuristics, and their size is small enough for our purposes. In the table below for each prime $p\in (61,199)\cup\{257\}\setminus \{79\}$ we give a list of elements of $\mathbb{F}_p$ that form an $S_1$-type subset and also its size which is smaller than $\sqrt{p-1}$ in each case.

\begin{center}
\begin{tabular}{ |c||c|c| } 
\hline
$p$ &  construction & size \\
\hline
67 & 0, 1, 3, 6, 11, 35, 54, 66 & 8  \\
\hline
71  & 0, 1,  3, 14, 36, 55, 63, 70 & 8 \\
 \hline
73 & 0, 1, 3, 6, 20, 37, 65, 72 & 8 \\
 \hline
83 & 0, 1,  2, 3, 7, 14, 44, 67, 82 & 9 \\
 \hline
89 & 0, 1, 2, 4, 8, 15, 47, 72, 88 & 9 \\
 \hline
97 & 0, 1,  2, 4, 8, 26, 51, 87, 96 & 9 \\
 \hline
101 & 0, 1, 2, 5, 9, 18, 56, 81, 100 & 9 \\
  \hline
103 & 0, 1, 2, 5, 10, 19, 59, 82, 102 & 9\\
  \hline
107 & 0, 1, 2, 4, 9, 29, 56, 96, 106 & 9 \\
  \hline
109 & 0, 1, 2, 5, 9, 29, 58, 98, 108 & 9 \\
  \hline
113 & 0, 1, 2, 5, 17, 29, 58, 94, 112 & 9\\
   \hline
127 & 0, 1, 2, 5, 17, 65, 96, 120, 126 & 9 \\
   \hline
131 & 0, 1, 3, 6, 11, 35, 70, 118, 130 & 9\\
   \hline
137 & 0, 1,  3, 6, 37, 62, 73, 87, 136 & 9 \\
   \hline
139  & 0, 1, 3, 6, 12, 38, 73, 125, 138 & 9\\
   \hline
149  & 0, 1, 3, 6, 20, 76, 113, 141, 148 & 9 \\
   \hline
151 & 0, 1, 3, 6, 20, 76, 115, 143, 150 & 9\\
   \hline
157 &  0, 1, 3, 6, 9, 15, 29, 89, 126, 156 & 10 \\
   \hline
163 & 0, 1, 3, 6, 9, 17, 28, 86, 133, 162 & 10\\
  \hline
167 & 0, 1, 3, 6, 9, 17, 31, 95, 134, 166 & 10 \\
   \hline
173 & 0, 1, 3, 6, 9, 17, 34, 59, 112, 172 & 10 \\
  \hline
179 & 0, 1, 3, 6, 9, 17, 34, 62, 115, 178 & 10 \\
   \hline
181 & 0, 1, 3, 6, 9, 17, 81, 134, 162, 180 &  10\\
 \hline
191 & 0, 1, 3, 6, 9, 17, 53, 100, 172, 190 & 10 \\
   \hline
193 & 0, 1, 3, 6, 9, 17, 34, 104, 157, 192 & 10 \\
  \hline
197 & 0, 1, 3, 6, 9, 17, 53, 106, 178, 196 & 10 \\
   \hline
257 & 0, 1, 3, 6, 9, 18, 35, 61, 87, 168, 256 & 11 \\
  \hline
\end{tabular}
\end{center}

\section{Acknowledgements} 
The authors would like to thank Noga Alon for pointing out the connections with  \cite{Alon-Ben} and \cite{Alweiss}.
The authors would like to thank Richárd Palincza for providing help in the computer calculations needed to find the constructions listed in the Appendix.

\end{document}